\documentclass[10pt]{amsart}

\usepackage{enumerate}
\usepackage{amssymb}
\usepackage{bm}
\usepackage{graphicx}
\usepackage[centertags]{amsmath}
\usepackage{amsfonts}
\usepackage{amsthm}
\usepackage{a4}
\usepackage{latexsym}
\usepackage{amsmath}\label{key}
\usepackage{amsfonts}
\usepackage{amssymb}
\usepackage{amscd}
\usepackage{amsthm}
\usepackage{layout}
\usepackage{color}
\usepackage{mathrsfs}

\usepackage[normalem]{ulem}%this allows to use \sout{} which draws a line right through bracketed text
\usepackage{comment}%allows to use \begin{comment} \end{comment} to comment out large parts of text
\theoremstyle{definition}
\newtheorem{pr}{Problem}
\newtheorem{thm}{Theorem}[section]
\newtheorem{cor}[thm]{Corollary}
\newtheorem{prop}[thm]{Proposition}
\newtheorem{lem}[thm]{Lemma}

\newtheorem*{thm*}{Theorem}
\newtheorem*{notn*}{Notation}

\theoremstyle{definition}
\newtheorem{dfn}[thm]{Definition}

\newtheorem{notn}[thm]{Notation}

\theoremstyle{definition}
\newtheorem{rem}[thm]{Remark}

\newcommand{\N}{\mathbb{N}}
\newcommand{\T}{\mathcal{T}}
\newcommand{\tT}{{\widetilde{\T}}}
\renewcommand{\tt}{{\tilde{t}}}
\newcommand{\ts}{{\tilde{s}}}
\newcommand{\tW}{{\widetilde{\W}}}
\newcommand{\M}{\mathcal{M}}

\newcommand{\W}{\mathcal{W}}
\renewcommand{\S}{\mathcal{S}}

\newcommand{\supp}{\textnormal{supp}}
\newcommand{\range}{\textnormal{range}}

\renewcommand{\-}{\textnormal{-}}

\DeclareRobustCommand{\rchi}{{\mathpalette\irchi\relax}}
\newcommand{\irchi}[2]{\raisebox{\depth}{$#1\chi$}}

\allowdisplaybreaks
\begin{document}

\title[non-asymptotic $\ell_1$ spaces with unique $\ell_1$ asymptotic model]{ non-asymptotic $\ell_1$ spaces with unique $\ell_1$ asymptotic model  }
%\dedicatory{}

\author[S. A. Argyros]{Spiros A. Argyros}
\address{National Technical University of Athens, Faculty of Applied Sciences,
Department of Mathematics, Zografou Campus, 157 80, Athens, Greece}
\email{sargyros@math.ntua.gr}

\author[A. Georgiou]{Alexandros Georgiou}
\address{National Technical University of Athens, Faculty of Applied Sciences,
Department of Mathematics, Zografou Campus, 157 80, Athens, Greece}
\email{ale.grgu@gmail.com}

\author[P. Motakis]{Pavlos Motakis}
\address{Department of Mathematics, University of Illinois at Urbana-Champaign, Urbana, IL 61801, U.S.A.}
\email{pmotakis@illinois.edu}

\thanks{The third author's research was supported by the National Science Foundation under Grant Numbers NSF DMS-1600600 and NSF DMS-1912897.}

\begin{abstract}
A recent result of Freeman, Odell, Sari, and Zheng  \cite{FOSZ} states that whenever a separable Banach space not containing $\ell_1$ has the property that all asymptotic models generated by weakly null sequences are equivalent to the unit vector basis of $c_0$ then the space is Asymptotic $c_0$. We show that if we replace $c_0$ with $\ell_1$ then this result is no longer true. Moreover, a stronger result of B. Maurey - H. P. Rosenthal \cite{MR} type is presented, namely, there exists a reflexive Banach space with an unconditional basis admitting $\ell_1$ as a unique asymptotic model whereas any subsequence of the basis generates a non-Asymptotic $\ell_1$ subspace.
\end{abstract}

\maketitle

\section{Introduction}
In this paper we study the question whether the uniqueness of asymptotic models, or equivalently, the uniform uniqueness of joint spreading models in a given Banach space implies that the space must be Asymptotic $\ell_p$. This is a coordinate free version from \cite{MMT} of the notion of asymptotic $\ell_p$ spaces with a Schauder basis by Milman and Tomczak-Jaegermann from \cite{MT}. The question draws its motivation from the following Problem of Halbeisen and Odell from \cite{HO} and a subsequent remarkable result from \cite{FOSZ}. Given a Banach space $X$, let $\mathscr{F}_0(X)$ denote the class of normalized weakly null sequences and $\mathscr{F}_b(X)$ denote the class of  all normalized block sequences of a fixed basis, if $X$ has one.

\begin{pr}[\cite{HO}]
\label{asmodelasellp}
Let $X$ be a Banach space that admits a unique asymptotic model with respect to $\mathscr{F}_0(X)$, or with respect to $\mathscr{F}_b(X)$ if $X$ has a basis. Does $X$ contain an Asymptotic $\ell_p$, $1\le p<\infty$, or an Asymptotic $c_0$ subspace?
\end{pr}

An asymptotic model is a notion which describes the asymptotic behavior of an array of sequences $(x_j^{i})_j$, $i\in\N$. On the contrary a space is Asymptotic $\ell_p$, for $1\le p<\infty$, (resp. Asymptotic $c_0$) if the asymptotic behavior of the whole space resembles that of $\ell_p$ (resp. $c_0$). Remarkably, in some cases unique asymptotic array structure implies that a space is Asymptotic $c_0$.

\begin{thm}[\cite{FOSZ}]\label{theorem FOSZ}
Let $X$ be a separable Banach space that does not contain $\ell_1$ and admits a unique $c_0$ asymptotic model with respect to $\mathscr{F}_0(X)$. Then the space $X$ is Asymptotic $c_0$.
\end{thm}

It was observed by Baudier, Lancien, Kalton, the third author, and Schlumprecht  in \cite[Section 9.2]{BLMS} that Theorem \ref{theorem FOSZ} no longer holds if we replace $c_0$ with $\ell_p$ for any $1<p<\infty$. The counterexamples are spaces very similar to the space defined by Szlenk in \cite{Sz}. The main purpose of this paper is to provide an answer for the remaining case $p=1$. Note that the main obstruction in this case is the fact that the $\ell_1$-norm is the largest one and hence, assuming that the space admits a unique $\ell_1$ asymptotic model which means a very strong presence of asymptotic $\ell_1$ structure, it is not obvious how to preserve a tree structure in the space which has norm smaller than $\ell_1$.
\begin{thm}
\label{main}
There exists a reflexive Banach space $X$ with an unconditional basis that admits a unique $\ell_1$ asymptotic model with respect to $\mathscr{F}_0(X)$, whereas it is not an Asymptotic $\ell_1$ space.
\end{thm}
In fact, for every countable ordinal $\xi$, there is a space $T^\xi_{inc}$, as in Theorem \ref{main}, that contains a weakly null $\ell_2$-tree of height $\omega^\xi$. An easy modification of $T^\xi_{inc}$ can yield a space containing a weakly null $\ell_p$-tree, for any $1<p<\infty$ with $p\neq2$, or a weakly null $c_0$-tree of height $\omega^\xi$.  Furthermore, the following analogue of the classical B. Maurey - H. P. Rosenthal \cite{MR} result is proved, which yields a stronger separation of the two properties than Theorem \ref{main}.

\begin{thm}
There exists a reflexive Banach space $X$ with an unconditional basis that admits a unique $\ell_1$ asymptotic model with respect to $\mathscr{F}_0(X)$, whereas any subsequence of the basis generates a non-Asymptotic $\ell_1$ subspace.
\end{thm}

More specifically, for every countable ordinal $\xi$, there is a space $T^\xi_{ess\-inc}$ as in the theorem above such that the space generated by any infinite subsequence of its basis contains a block $c_0$-tree of height $\omega^\xi$. It is possible to modify $T^\xi_{ess\-inc}$ to contain $\ell_p$-trees for any $1<p<\infty$, instead of $c_0$-trees.

In the final part of this paper we show that, for $1<r<p<\infty$, certain spaces ${JT}^\xi_{r,p}$, similar to those  defined by Odell and Schlumprecht in \cite[Example 4.2]{OS3} (see also \cite[page 66]{O2}), admit a unique $\ell_p$ asymptotic model but are not Asymptotic $\ell_p$.  These are spaces with an unconditional Schauder basis $(e_t)_{t\in\mathcal{\T_\xi}}$ indexed over a well-founded and infinite branching countable tree $\T_\xi$ of height $\omega^\xi$. The norm of ${JT}^\xi_{r,p}$ is defined as follows: if $x = \sum_{t\in\mathcal{T}_\xi}a_te_t$ and $S$ is a segment of $\mathcal{T}_\xi$ define $\|S(x)\|^r_r = \sum_{t\in S}|a_t|^r$ and

\begin{equation}\label{jtqp}\tag{1.8}
\|x\|_{{JT}_{r,p}^\xi} = \sup\bigg\{\Big(\sum_{i=1}^n\|S_i(x)\|_r^p\Big)^{1/p}\!\!\!\!: (S_i)_{i=1}^n \text{  disjoint segments of }\T_\xi\bigg\}.\end{equation}

The space $T^\xi_{inc}$ from Theorem \ref{main} is defined on the same tree. We say that two segments $S_1$, $S_2$ of $\mathcal{T}_\xi$ are incomparable if any node of $S_1$ is incomparable to any node of $S_2$. We relabel the basis of the Tsirelson space $T$ as $(e_t)_{t\in\mathcal{T}_\xi}$ so that the order is compatible with the initial one and define the norm of $T^\xi_{inc}$ as follows : for $x = \sum_{t\in\mathcal{T}_\xi}a_te_t$ define $\|S(x)\|^2_2 = \sum_{t\in S}|a_t|^2$ and

\[\|x\|_{T_{inc}^{\xi}} = \sup\bigg\{\Big\|\sum_{i=1}^n\|S_i(x)\|_{2}e_{\min{S_i}}\Big\|_T\!\!\!: (S_i)_{i=1}^n \text{ incomparable segments of } \T_\xi\bigg\}.\]
However, we will not use the above description of the norms. Instead we revert to the notion of norming sets and norming functionals. This makes some parts of the proof easier and it can also be potentially useful to show similar results on more complicated spaces based on these norms.

Finally, we should mention that Problem \ref{asmodelasellp} is only one of several concerning the separation of different asymptotic structures in Banach space theory. For example, in \cite{AM3} the first and third author showed that there exist spaces with a uniformly unique spreading model, which can be chosen to be any $\ell_p$ or $c_0$, that have no Asymptotic $\ell_p$ or $c_0$ subspace. This answers a question by Odell in \cite{O1} and Junge, Kutzarova, and Odell in \cite{JKO}. Moreover, in \cite{KM}  Kutzarova and the third author showed that certain spaces by Beanland, the first author, and the third author from  \cite{ABM} are asymptotically symmetric and have no Asymptotic $\ell_p$ or $c_0$ subspaces, answering a question from \cite{JKO}.

    \begin{notn*}
	By $\N=\{1,2,\ldots\}$ we denote the set of all positive integers. We will use capital letters as $L,M,N,\ldots$ (resp. lower case letters as $s,t,u,\ldots$) to denote infinite subsets (resp. finite subsets) of $\N$. For every infinite subset $L$ of $\N$, the notation $[L]^\infty$ (resp. $[L]^{<\infty}$) stands for the set of all infinite (resp. finite) subsets of $L$. For every $s\in[\N]^{<\infty}$, by $|s|$ we denote the cardinality of $s$. For $L\in[\N]^\infty$ and $k\in\N$, $[L]^k$ (resp. $[L]^{\le k}$) is the set of all $s\in[L]^{<\infty}$ with $|s|=k$ (resp. $|s|\le k$). For every $s,t\in[\N]^{<\infty}$, we write $s<t$ if either at least one of them is the empty set, or $\max s<\min t$. Also for  $\emptyset\neq M\in[\N]^\infty$ and $n\in\N$ we write $n<M$ if $n<\min M$. For $s=\{n_1<\ldots<n_k\}\in[\N]^{<\infty}$ and for each $1\le i\le k$, we set $s(i)=n_i$.
	
	Moreover, we follow \cite{LT} for standard notation and terminology concerning Banach space theory.
\end{notn*}

\section{Asymptotic structures}

Let us recall the definitions of the asymptotic notions that appear in the results of this paper and were mentioned in the introduction. Namely, asymptotic models, joint spreading models and the notions of Asymptotic $\ell_p$ and Asymptotic $c_0$ spaces. For a more thorough discussion, including several open problems and known results, we refer the reader to \cite[Section 3]{AM3}.

\begin{dfn}[\cite{HO}]
An infinite array of sequences $(x^{i}_j)_j$, $i\in\N$, in a Banach space $X$, is said to generate a sequence $(e_i)_i$, in a seminormed space $E$, as an asymptotic model if for every $\varepsilon>0$ and $n\in\N$, there is a $k_0\in\N$ such that for any natural numbers $k_0\leq k_1<\cdots<k_n$ and any choice of scalars $a_1,\ldots,a_n$ in $[-1,1]$ we have that
\[\Bigg|\Big\|\sum_{i=1}^na_ix_{k_i}^{i}\Big\| - \Big\|\sum_{i=1}^na_ie_{i}\Big\|\Bigg| < \varepsilon.\]
\end{dfn}

A Banach space $X$ is said to admit a unique asymptotic model with respect to a family $\mathscr{F}$ of normalized sequences in $X$ if whenever two infinite arrays, consisting of sequences from $\mathscr{F}$, generate asymptotic models then those must be equivalent. Typical families under consideration are those of normalized weakly null sequences, denoted  $\mathscr{F}_0(X)$, normalized Schauder basic sequences, denoted  $\mathscr{F}(X)$, or the family all normalized block sequences of a fixed basis of $X$, if it has one, denoted  $\mathscr{F}_b(X)$.

The notion of plegma families was first introduced by Kanellopoulos, Tyros, and the first author  in \cite{AKT}. We will use the slightly modified definition of \ from \cite{AGLM}.

	\begin{dfn}[\cite{AGLM}]
	Let $M\in[\N]^\infty$ and $k\in\N$. A {plegma} (resp. {strict plegma}) family in $[M]^k$ is a finite sequence $(s_i)_{i=1}^l$ in $[M]^k$ satisfying the following.
	\begin{enumerate}
		\item[(i)] $s_{i_1}(j_1)<s_{i_2}(j_2)$ for every $1\le j_1<j_2\le k$  and $1\le i_1,i_2\le l$.
		\item[(ii)] $s_{i_1}(j)\le s_{i_2}(j)$ $\big($resp. $s_{i_1}(j)< s_{i_2}(j)\big)$ for all $1\le i_1<i_2\le l$ and $1\le j\le k$ .
	\end{enumerate}
	For each $l\in \N$, the set of all sequences $(s_i)^l_{i=1}$ which are plegma families in $[M]^k$ will be denoted by $Plm_l([M]^k)$ and that of the strict plegma ones by $S$-$Plm_l([M]^k)$.
\end{dfn}

\begin{dfn}[\cite{AGLM}]
A finite array of sequences $(x^{i}_j)_j$, $1\leq i\leq l$, in a Banach space $X$, is said to generate another array of sequences $(e_j^{i})_j$, $1\leq i\leq l$, in a seminormed space $E$, as a joint spreading model if for every $\varepsilon>0$ and $n\in\N$, there is a $k_0\in\N$ such that for any $(s_i)_{i=1}^l\in S$-$Plm([\N]^n)$ with $k_0\le s_1(1)$ and any $l\times n$ matrix $A=(a_{ij})$ with entries in $[-1,1]$ we have that
\[\Bigg|\Big\|\sum_{i=1}^l\sum_{j=1}^na_{ij}x_{s_i(j)}^{i}\Big\| - \Big\|\sum_{i=1}^l\sum_{j=1}^na_{ij}e_j^{i}\Big\|\Bigg|<\varepsilon.\]
\end{dfn}

 A Banach space $X$ is said to admit a  uniformly unique joint spreading model with respect to a family of normalized sequences $\mathscr{F}$ in $X$, if there exists a constant $C$ such that whenever two arrays $(x_j^{i})_j$ and $(y_j^{i})_j$, $1\leq i\leq l$, of sequences from $\mathscr{F}$ generate joint spreading models then those must be $C$-equivalent. Moreover, a Banach space admits a uniformly unique joint spreading model with respect to a family $\mathscr{F}$ if and only if it admits a unique asymptotic model with respect to $\mathscr{F}$ (see, e.g.,  \cite[Remark 4.21]{AGLM} or \cite[Proposition 3.12]{AM3}).

It was proved in \cite{AGLM} that whenever a Banach space admits a uniformly unique joint spreading model with respect to some family satisfying certain stability conditions, then it satisfies a property concerning its bounded linear operators, called the Uniform Approximation on Large Subspaces property (see \cite[Theorem 5.17]{AGLM} and \cite[Theorem 5.23]{AGLM}).

 \begin{dfn}[\cite{MT} and \cite{MMT}]
A Banach space $X$ is called Asymptotic $\ell_p$,  $1\leq p<\infty$, (resp. Asymptotic $c_0$) if there exists a constant $C$ such that in a two-player $n$-turn game $G(n,p,C)$, where in each turn $k=1,\ldots,n$ player (S) picks a finite codimensional subspace $Y_k$ of $X$ and then player (V) picks a normalized vector $x_k\in Y_k$, player (S) has a winning strategy to force player (V) to pick a sequence $(x_k)_{k=1}^n$ that is $C$-equivalent to the unit vector basis of $\ell^n_p$ (resp. $\ell^\infty_n)$.
\end{dfn}

The typical example of a non-classical Asymptotic $\ell_p$ space is the Tsirelson space as defined by Figiel and Johnson in \cite{FJ}. This is a reflexive Asymptotic $\ell_1$ space and it is the dual of Tsirelson's original space from \cite{T} that is  Asymptotic $c_0$. Finally, whenever a Banach space is Asymptotic $\ell_p$ or Asymptotic $c_0$, then it admits a uniformly unique joint spreading model with respect to $\mathscr{F}_0(X)$ (see, e.g., \cite[Corollary 4.12]{AGLM}).

\section{A family of non-asymptotic $\ell_1$ spaces admitting uniformly unique $\ell_1$ joint spreading models}
\label{ell1section}
In this section we define the spaces $T_{inc}^\xi$, for each countable ordinal $\xi$, and we prove that they admit a uniformly unique $\ell_1$ joint spreading model with respect to $\mathscr{F}_b(T_{inc}^\xi)$ and are not Asymptotic $\ell_1$. The spaces are defined in terms of norming sets and norming functionals as this is  more convenient to prove the desired result.
	
	\subsection{Measures on Well-Founded Countable Compact Trees} We start with a key result that will be used later to prove that $T^\xi_{inc}$ admits a uniformly unique joint spreading model equivalent to the unit vector basis of $\ell_1$.
	
	 Let $\preceq$ be a partial order on some infinite subset $M$ of the naturals, which is compatible with the standard order, i.e. $n\preceq m$ implies $n\le m$, for all $n,m\in M$. Assume that, for each $n\in M$, the set $S_n=\{m\in M: m\preceq n \}$ is finite and totally ordered with respect to $\preceq$, that is, $\T=(M,\preceq)$ is a tree. Let us also assume that the tree $\T$ is well-founded, i.e., it contains no infinite totally ordered sets, and infinite branching, i.e., every non-maximal node has infinitely many immediate successors.
	
	 Observe that $\tT=(\{S_t :{t\in \T}\},\subset)$ is also a tree and is in fact isomorphic to $\T$ via the mapping $t\mapsto S_t$. Given $t\in\T$, we will denote $S_t$ by $\tt$. Moreover, two nodes  $\tt_1,\tt_2$ are incomparable in $\tT$ if and only if the nodes $t_1,t_2$ are incomparable in $\T$, i.e. not comparable in the respective order. For $\tt\in\T$, we denote by $V_{\tt}$ the set consisting of $\tt$ and all of its successors.
	
	 Note that $\tT$ is a countable compact space when equipped with the pointwise convergence topology and hence $\mathcal{M}(\tT)$, the set of all regular  measures on $\tT$, is isometric to $\ell_1(\tT)$. In particular, each $\mu\in\mathcal{M}(\tT)$ is of the form $\mu=\sum_{\tt\in \tT}a_\tt\delta_\tt$, where $\delta_\tt$ is the Dirac measure centered on $\tt$, and $\|\mu\|=\sum_{\tt\in \tT}|a_\tt|$. Finally, the support of $\mu$ is defined as $\supp\mu=\{\tt\in \tT:a_\tt\neq0 \}$.  We will prove the following proposition, starting with Lemma \ref{lemma incomparable}

	\begin{prop}\label{proposition incomparable}
				Let $(\mu_j)_j$ be a sequence of positive regular measures on $\tT$ with disjoint finite supports and let $c>0$ be such that $\mu_j(\tT)<c$ for all $j\in\N$. Then, for every $\varepsilon>0$, there is an $L\in[\N]^\infty$ and a subset $G_j$ of $\supp\mu_{j}$ for each $j\in L$,  satisfying the following.
		\begin{itemize}
			\item[(i)] $\mu_{j}(\tT\setminus G_j)\le \varepsilon$ for every $j\in L$.
			\item[(ii)] The sets $G_j$, $j\in\N$, are pairwise incomparable.
		\end{itemize}
	\end{prop}

	\begin{lem}\label{lemma incomparable}
		Let $(\mu_j)_j$ be a sequence of positive regular measures on $\tT$ with disjoint finite supports and let $c>0$ be such that $\mu_j(\tT)<c$ for all $j\in\N$. Assume that $w^*\-\lim_j\mu_j=\mu=\sum_{\tt\in \tT}a_\tt\delta_\tt$. Then, for every $\tt\in\supp\mu$ and $\varepsilon>0$, there is an $L\in[\N]^\infty$ and a subset $G^\tt_j$ of $\supp\mu_{j}$ for each $j\in L$,  satisfying the following.
		\begin{itemize}
			\item[(i)] $G^\tt_j\subset V_\tt$ for every $j\in L$.
			\item[(ii)] $|\mu_{j}(G^\tt_j)-a_\tt|<\varepsilon$ for every $j\in L$.
			\item[(iii)] The sets $G^\tt_j$, $j\in L$, are pairwise incomparable.
		\end{itemize}
	\end{lem}
	\begin{proof} Recall that the nodes of $\T$ are in fact naturals numbers. Hence identifying $\{t:\tt\in\supp \mu_j \}$, $j\in\N$, as subsets of the naturals and passing to a subsequence, we may assume that they are successive.
	
	Let $(\tt_j)_j$ be an enumeration of the immediate successors of $\tt$ and for each $j\in \N$ define $W^\tt_j=V_\tt\setminus\cup_{i=1}^jV_{\tt_i}$. Observe that $(W^\tt_j)_j$ is a decreasing sequence of clopen subsets of $\tT$ with $\cap_{j}W^\tt_j=\{\tt\}$ and hence $\lim_j\mu(W^\tt_j)=a_\tt$ and $\lim_j\mu_j(W^\tt_i)=\mu(W^\tt_i)$ for all $i\in\N$. We can thus find $N\in[\N]^\infty$ and pass to a subsequence of $(\mu_{j})_j$, which we relabel for convenience, so that $\lim_{j\in N}|\mu_{j}(W^\tt_j)-\mu(W^\tt_j)|=0$ and define $G^\tt_j=\supp\mu_{j}\cap W^\tt_j$ for each $j\in N$. Note then that $\lim_{j\in N}\mu_{j}(G^\tt_j)=a_\tt$ and $\mu_{j}|_{G^\tt_j}(\cup_{i=1}^j V_{\tt_i})=0$ for all $j\in N$.
		
		There is at most one $j\in N$ such that $\tt\in G^\tt_j$ and hence, passing to a subsequence, we may assume that $\tt\notin G^\tt_j$ for all $j\in N$. Moreover, since $\lim_{j\in N}\mu_{j}(G^\tt_j)=a_\tt$, we may even pass to a further subsequence such that $|\mu_{j}(G^\tt_j)-a_\tt|<\varepsilon$ for all $j\in N$. For the remaining part of the proof we will choose, by induction, an $L\in[N]^\infty$ such that $G^\tt_{j}$, $j\in L$, are pairwise incomparable. Set $l_1=\min N$ and suppose that we have chosen $l_1<\ldots<l_k$ in $N$, for some $k\in\N$, such that $G^\tt_{l_i}$ and $G^\tt_{l_j}$ are incomparable, $1\le i<j\le k$. Pick $l_k<l_{k+1}\in N$ such that $\mu_{l_{k+1}}|_{G^\tt_{l_{k+1}}}(\cup\{ V_\ts:\ts\in\cup_{i=1}^k G^\tt_{l_i}\} )=0$. Then, if for some $1\le i \le k$ the nodes $\ts_1\in G^\tt_{l_i}$ and $\ts_2\in G^\tt_{l_{k+1}}$ are comparable, we have that $\ts_2\in V_{\ts_1}$ whereas $\mu_{l_{k+1}}(V_{\ts_{1}})=0$, which is a contradiction. Hence $G^\tt_{l_{1}},\ldots,G^\tt_{l_{k+1}}$ are pairwise incomparable.
	\end{proof}

	\begin{proof}[Proof of Proposition \ref{proposition incomparable}]
		 Passing to a subsequence, since $\tT$ is compact with respect to the pointwise convergence topology and $(\mu_j)_j$ are uniformly bounded, we may assume that $(\mu_j)_j$ $w^*$-converges to some measure $\mu=\sum_{\tt\in\tT} a_\tt\delta_\tt$ in $\mathcal{M}(\tT)$.
		
		 Let $\delta>0$ be such $(1-\delta)(\mu(\tT)-\delta)> \mu(\tT)-{\varepsilon}/{2}$ and pick $n_0\in\N$ such that $\sum_{i=1}^{n_0}a_{\tt_i}\ge \mu(\tT)-\delta$. Applying the previous lemma successively for each $\tt_i$, $i=1,\ldots,n_0$, we obtain an $L\in[\N]^\infty$ and, for each $j\in L$ and $i=1,\ldots,n_0$, a subset $G^i_j$ of $\supp\mu_{j}$ satisfying items (i) - (iii) of Lemma \ref{lemma incomparable} for $\tt_i$ and $\delta a_{\tt_i}$. Note that if $\tt_{i_1},\tt_{i_2}$ are incomparable for some $1\le i_1,i_2\le n_0$, then by item (i), the sets $G^{i_1}_{j_1}$ and $G_{j_2}^{i_2}$ are also incomparable for any $j_1,j_2\in L$. If the nodes $\tt_{i_1},\tt_{i_2}$ are comparable, say $\tt_{i_1}\subset\tt_{i_2}$, then  there exists at most one $j\in L$ such that $\tt_{i_2}\in G^{i_1}_j$. Hence by a finite induction argument, we may pass to a subsequence such that the sets $G^i_j$, for $i=1,\ldots,n_0$ and $j\in L$, are pairwise incomparable.  Define $G_j=\cup_{i=1}^{n_0}G^i_j$, $j\in L$, and conclude that
\[
\mu_{j}(G_j)=\sum_{i=1}^{n_0}\mu_{j}(G^i_j)\ge\sum_{i=1}^{n_0}a_{\tt_i}-\delta a_{\tt_i}\ge(\mu(\tT)-\delta)(1-\delta)> \mu(\tT)-\frac{\varepsilon}{2}.
\]
Finally, passing to a further subsequence if necessary, we may also assume that $|\mu_j(\tT)-\mu(\tT)|<{\varepsilon}/{2}$ and hence $\mu_j(\tT\setminus G_j)<\varepsilon$ for every $j\in L$.
\end{proof}

\subsection{Tsirelson Extension of a Ground Set.} In order to define $T_{inc}^\xi$, we first introduce some necessary concepts used in the construction of Tsirelson type spaces.

\begin{dfn}
	A subset $W$ of $c_{00}(\N)$ is called a norming set if it satisfies the following conditions.
	\begin{itemize}
		\item[(i)] $W$ is symmetric and $ e^*_i\in W$ for every $i\in\N$.
		\item[(ii)] $\|f\|_\infty\le1$ for every $f\in W$.
		\item[(iii)] $W$ is closed under the restriction of its elements to intervals of $\N$.
	\end{itemize}
\end{dfn}
A norming set $W$ induces a norm $\|\cdot\|_W$ on $c_{00}(\N)$ defined as
\[
\|x\|_W=\sup\{f(x):f\in W\}.
\]

\begin{dfn}Let $G$ be a norming set on $c_{00}(\N)$. The Tsirelson extension of $G$, denoted by $W_G$, is the minimal subset of $c_{00}(\N)$ that contains $G$ and is closed under the $(\mathcal{S},1/2)$-operation, i.e., if  $f_1,\ldots,f_n$ are in $W_G$ and $n\le \supp f_1<\ldots<\supp f_n$, then $1/2\sum_{i=1}^nf_i$ is also in $W_G$. We call $G$ the ground set of $W_G$.
\end{dfn}

Note that $W_G$ is a norming set on $c_{00}(\N)$. Moreover, the induced norm $\|\cdot\|_{W_G}$ satisfies the following implicit equation
\[
\|x\|_{W_G}=\max\Big\{\|x\|_G,\frac{1}{2}\sup\sum_{i=1}^n\|E_ix\|_{W_G} \Big\}
\]
where the supremum is taken over all finite collections $E_1,\ldots,E_n$ of successive intervals of $\N$ with $n\le E_1$.

\begin{dfn}
	Let $f\in W_G$. For a finite tree $\mathcal{A}$, a family $(f_\alpha)_{\alpha\in \mathcal{A}}$ is said to be a tree analysis of $f$ if the following are satisfied.
	\begin{itemize}
		\item[(i)] $\mathcal{A}$ has a unique root denoted by 0 and $f_0=f$.
		\item[(ii)] For every maximal node $\alpha\in\mathcal{A}$  we have that $f_\alpha\in G$.
		\item[(iii)] Let $\alpha$ be a non-maximal node of $\mathcal{A}$ and denote by $S(\alpha)$ set of  immediate successors of $\alpha$. Then $f_\alpha\in W_G$ and the ranges of $f_s$, $s\in S(\alpha)$, are disjoint and $f_\alpha=1/2\sum_{s\in S(\alpha)}f_s$.
	\end{itemize}
\end{dfn}

It follows, by minimality, that every $f\in W_G$ admits a tree analysis.

\begin{prop}\label{proposition tree analysis}
	Let $f\in W_G$ with a tree analysis $(f_\alpha)_{\alpha\in \mathcal{A}}$ and denote by $\mathcal{M}$ the set of all maximal nodes of $\mathcal{A}$. Then the following hold.
	\begin{itemize}
		\item[(i)] For every $\alpha\in\mathcal{M}$, there is a $k_\alpha\in\N\cup\{0\}$ such that $f=\sum_{\alpha\in \mathcal{M}}f_\alpha/2^{k_\alpha}$.
		\item[(ii)] If $\mathcal{N}\subset\mathcal{M}$, then $g=\sum_{\alpha\in \mathcal{N}}f_\alpha/2^{k_\alpha}$ is in $W_G$ and $g=f|_{\cup\{\supp f_\alpha:\alpha\in \mathcal{N}\}}$.
	\end{itemize}
\end{prop}

For an extensive review on Tsirelson's space we refer the reader to \cite{CS}.

\subsection{Definition of the space $T^\xi_{inc}$}

We define the space $T^\xi_{inc}$ as the completion of $c_{00}(\N)$ with respect to the norm induced by a norming set $W_\xi$. This norming set is a subset of the Tsirelson extension of a ground set $G^\xi_2$, the functionals of which satisfy a certain property. Both this property and $G^\xi_2$ are defined via an infinite branching well-founded tree $\T_\xi$ on the natural numbers.

We start by fixing a partition of the naturals $\N=\cup_{j=0}^\infty N_j$ into infinite sets and an injection $\phi:[\N]^{<\infty}\to\N$. Recall the definition of the Schreir families $(\S_\xi)_{\xi<\omega_1}$.

\begin{dfn}
Let $\xi$ be a countable ordinal. We define, by transfinite induction, the Schreier family $\mathcal{S}_\xi\subset[\N]^{<\infty}$  as follows.
\begin{itemize}
	\item[(i)] If $\xi=0$, then $\mathcal{S}_0=\{\{n\}:n\in\N \}\cup\{\emptyset\}$.
	\item[(ii)] If $\xi=\alpha+1$, then \[\mathcal{S}_\xi =\{\cup_{j=1}^nE_j:n\in\N,\; E_1<\ldots<E_n\in\S_\alpha\text{ and }n\le E_1 \}.\]
	\item[(iii)] If $\xi$ is a limit ordinal we choose a fixed sequence $(\alpha(\xi,j))_j\subset[1,\xi)$ which increases to $\xi$ and set \[\mathcal{S}_\xi=\{E\subset \N:\text{ there exists } j\in\N \text{ such that }E\in S_{\alpha(\xi,j)}\text{ and }j\le E\}.\]
\end{itemize}
\end{dfn}

We now define the tree $\T_\xi$, by defining a partial order $\preceq_\xi$ on $\N$.

\begin{dfn}
Fix a countable ordinal $\xi$ and define the partial order $\preceq_\xi$ on $\N$ as follows: $n\preceq_\xi m$ if there exists $\{n_0,\ldots,n_k \}\in \mathcal{S}_\xi$ such that
\begin{itemize}
\item[(i)] $n_0\in N_0$ and $n_i\in N_{\phi(n_0,\ldots,n_{i-1})}$ with $n_{i-1}<n_i$ for every $1\le i \le k$,
\item[(ii)]  $n=n_i$ and $m=n_j$ for some $0\le i\le j\le k$.
\end{itemize}
\end{dfn}

\begin{rem} Note that $\T_\xi=(\N,\preceq_\xi)$ is an infinite branching tree and it is also well-founded since $\mathcal{S}_\xi$ is a compact family, i.e., $\{\rchi_E:E\in\S_\xi\}$ is a compact subset of $\{0,1\}^\N$. Moreover, the partial order $\preceq_\xi$ is compatible with the standard order on the naturals and finally, standard inductive arguments yield that $\T_\xi$ is of height $\omega^\xi$.
\end{rem}

\begin{dfn} Define the following norming set on $c_{00}(\N)$
\[
{G_2^{\xi}}=
\Big\{ \sum_{i\in S}a_ie^*_i:S\text{ is a segment of }\T_\xi\text{ and }\sum_{i\in S}a_i^2\le 1\Big\}
\]
and denote by $W_\xi$ the subset of $W_{G^\xi_2}$ containing all $f$ with tree analysis $(f_\alpha)_{\alpha\in\mathcal{A}}$ such that there exist pairwise incomparable segments $S_{\alpha}$ of $\T_\xi$ with $\supp f_\alpha\subset S_{\alpha}$ for every maximal node $\alpha\in\mathcal{A}$.
Denote by $T_{inc}^\xi$ the completion of $c_{00}(\N)$ with respect to the norm $\|\cdot\|_{W_\xi}$ induced by the norming set $W_\xi$.
\end{dfn}

\begin{rem}
	The unit vector basis $(e_j)_j$ of $c_{00}(\N)$ forms a 1-unconditional Schauder basis for the space $T^\xi_{inc}$. Moreover it is boundedly complete, since $T^\xi_{inc}$ admits $\ell_1$ as a uniformly unique spreading model as shown in Proposition \ref{first space unique as model}.
\end{rem}

First, we show that $T_{inc}^\xi$ admits a uniformly unique joint spreading model with respect to $\mathscr{F}_b(T_{inc}^\xi)$, that is equivalent to the unit vector basis of $\ell_1$.

\begin{prop}\label{first space unique as model}
	The space $T_{inc}^\xi$ admits a uniformly unique joint spreading model with respect to $\mathscr{F}_b(T_{inc}^\xi)$, which is equivalent to the unit vector basis of $\ell_1$.
\end{prop}
\begin{proof}
	Let $(x^i_j)_j$, $1\le i\le l$, be an array of normalized block sequences in $T_{inc}^\xi$ and $\varepsilon>0$. Passing to a subsequence, we assume that $\supp x^{i_1}_j <\supp x^{i_2}_{j+1}$ for every $i_1,i_2=1,\ldots, l$ and $j\in\N$. For every $i=1,\ldots,l$ and $j\in\N$, pick a functional $f^i_j=\sum_{\alpha\in\mathcal{M}_{j}^i }f^i_{j,\alpha}/2^{k^i_{j,\alpha}}$ in $W_\xi$ such that $f^i_j(x^i_j)\ge1-\varepsilon$ and $f^i_{j,\alpha}(x^i_j)>0$ for every $\alpha\in \mathcal{M}^i_j$, where $\mathcal{M}^i_j$ denotes the set of all maximal nodes of a fixed tree analysis of $f^i_j$. For every $i=1,\ldots,l$, $j\in\N$ and $\alpha\in \mathcal{M}^i_j$, define $\lambda^i_{j,\alpha}=f^i_{j,\alpha}(x^i_j)/2^{k^i_{j,\alpha}}$ and  $t^i_{j,\alpha}=\min\supp f^i_{j,\alpha}$. Moreover, for each $j\in\N$, define the probability measure \[\mu_j=\frac{1}{l}\sum_{i=1}^l\frac{1}{f^i_j(x^i_j)}\sum_{\alpha\in \mathcal{M}_{j}^i}\lambda^i_{j,\alpha}\delta_{\tt^i_{j,\alpha}}.\] Then,  Proposition \ref{proposition incomparable} yields an $L\in[\N]^\infty$ and a sequence $(G_j)_{j\in L}$ of pairwise incomparable subsets of ${\tT_\xi}$ such that $\mu_{j}(G_j)\ge 1-\delta$ for every $j\in L$ and for some $\delta$ sufficiently small such that for any $i=1,\ldots,l$ and $j\in L$
	\begin{equation}\label{each measure on Gj}\tag{2.8.1}
	\frac{1}{f_j^i(x_j^i)}\sum_{\alpha\in \mathcal{M}_{j}^i}\lambda^i_{j,\alpha}\delta_{\tt^i_{j,\alpha}}(G_j)\ge (1-\varepsilon)^2.
	\end{equation}
	
	Let $k\in \N$ and $(s_i)_{i=1}^l\in S$-$Plm_l([L]^k)$ with $kl\le x_{s_1(1)}$. Then, for  $i=1,\ldots,l$ and $j\in L$, if $\mathcal{N}^i_j=\{\alpha \in\mathcal{M}^i_{s_i(j)}:\tt^i_{j,\alpha}\in G_{s_i(j)  }\}$, item (ii) of Proposition \ref{proposition tree analysis} yields that
	\[g^i_j\;=\sum_{\alpha\in\mathcal{N}^i_j } \frac{1}{2^{k^i_{s_i(j),\alpha}}}f^i_{s_i(j),\alpha}\in W_{\xi}.\] Moreover, \eqref{each measure on Gj} implies  $g^i_j(x^i_{s_i(j)})\ge (1-\varepsilon)^2$ for all $i=1,\ldots,l$ and $j\in L$, and since $G_j$, $j\in L$, are pairwise incomparable, we have  that $g=1/2\sum_{i=1}^l\sum_{j=1}^kg^i_j$ is in $W_{{\xi}}$. Then for any choice of scalars $(a_{ij})_{i=1,j=1}^{l,k}$, due to unconditionality, we conclude that
	\[
	\Big\|\sum_{i=1}^l\sum_{j=1}^ka_{ij}x^i_{s_i(j)}\Big\|\ge \Big\|\sum_{i=1}^l\sum_{j=1}^k|a_{ij}|x^i_{s_i(j)}\Big\|\ge\frac{(1-\varepsilon)^2}{2}\sum_{i=1}^l\sum_{j=1}^k|a_{ij}|.\]
\end{proof}

\begin{prop}\label{section 1 reflexivity}
	The space $T_{inc}^\xi$ is reflexive.
\end{prop}
\begin{proof}
	Since $T_{inc}^\xi$ admits a boundedly complete unconditional Schauder basis, it does not contain $c_0$ (see \cite[Theorem 1.c.10]{LT}) and hence it suffices to show that it does not contain $\ell_1$ as follows from \cite[Theorem 2]{J1}.
	
	Fix $n\in\N$. Let $(x_j)_j$ be a normalized block sequence in $T_{inc}^\xi$ and $f=\sum_{i\in S}b_ie^*_i$  in $G^\xi_2$. For each $j=1,\ldots,n$, define $I_j=\{i\in S: i\in\supp x_j \}$ and note that
	\[
	\Big(\sum_{i\in I_j}b_jx_j(i)\Big)^2\le\sum_{i\in I_j}b_i^2.
	\]	
	Then, for any choice of scalars $a_1,\ldots,a_n$, we have that
	\begin{align*}
	f\Big(\sum_{j=1}^na_jx_j\Big)&=\sum_{j=1}^na_j\sum_{i\in S}b_ix_j(i)\le\Big(\sum_{j=1}^na_j^2\Big)^{\frac{1}{2}}\Big(\sum_{j=1}^n\Big(\sum_{i\in I_j}b_ix_j(i)\Big)^2\Big)^{\frac{1}{2}}\\
	&\le \Big(\sum_{j=1}^na_j^2\Big)^{\frac{1}{2}}\Big(\sum_{j=1}^n\sum_{i\in I_j}b_i^2\Big)^{\frac{1}{2}}\le \Big(\sum_{j=1}^na_j^2\Big)^{\frac{1}{2}}
	\end{align*}
	and hence
	\[
	\Big\|\sum_{j=1}^{n}a_jx_j\Big\|_{G^\xi_2}\le \Big(\sum_{j=1}^na_j^2\Big)^{\frac{1}{2}}.
	\]	
	That is, for any normalized block sequence $(x_j)_j$ in $T_{inc}^\xi$, there exists a block subsequence $(y_j)_j$ with $\|y_j\|_{G_2^\xi}\to0$.
	
We show that $T^\xi_{inc}$ does not contain $\ell_1$ in a similar manner as in the proof of the reflexivity for the classical Tsirelson space \cite{FJ}.	 Suppose that $T^\xi_{inc}$ contains $\ell_1$. Then James' $\ell_1$ distortion theorem \cite{J2} implies that, for $\varepsilon<{1}/{4}$, there exists a normalized block sequence $(x_j)_j$ in $T^\xi_{inc}$ such that
	 \[
	 \Big\|\sum_{j=1}^na_jx_j\Big\|\ge (1-\varepsilon)\sum_{j=1}^n|a_j|
	 \]
	 for any $n\in\N$ and any choice of scalars $a_1,\ldots,a_n$. Applying the result of the previous paragraph, we may also assume that $\|x_j\|_{G^\xi_2}<{1}/{2}$ for every $j\in\N$ and hence, for any $n\ge2$, we have that
	 \begin{equation}\tag{2.9.1}\label{reflexivity they are not normed by ground set}
	 \Big\|x_1+\frac{1}{n}\sum_{i=2}^{n+1}x_i\Big\|> \Big\|x_1+\frac{1}{n}\sum_{i=2}^{n+1}x_i\Big\|_{G^\xi_2}.
	 \end{equation}
	 Moreover, for any $n\in\N$, we have that
	 \[
	  \Big\|x_1+\frac{1}{n}\sum_{i=2}^{n+1}x_i\Big\|\ge 2(1-\varepsilon).
	 \]
	 Observe that (\ref{reflexivity they are not normed by ground set}) implies that there exists $f={1}/{2}\sum_{j=1}^kf_j\in W_\xi\setminus G^\xi_2$ such that
	 \[
f\Big(x_1+\frac{1}{n}\sum_{i=2}^{n+1}x_i\Big)>	  \Big\|x_1+\frac{1}{n}\sum_{i=2}^{n+1}x_i\Big\|-{\varepsilon} \ge\frac{5}{4}
 \]
 and that $\min\supp f_1\le \max\supp x_1$, since otherwise
 \[
 f\Big(x_1+\frac{1}{n}\sum_{i=2}^{n+1}x_i\Big)=\frac{1}{n}\sum_{i=2}^{n+1}f(x_i)\le 1.
 \]
 Therefore, $k\le \max\supp x_1$. Note that there are at most $k$ $i$'s such that the support of $x_i$ intersects the supports of at least two $f_j$'s and hence
 \[
f\Big(x_1+\frac{1}{n}\sum_{i=2}^{n+1}x_i\Big)\le 1 +\frac{k}{n}+\frac{n-k}{2n}\le 1+\frac{n+\max\supp x_1}{2n}\xrightarrow[n \to \infty]{}\frac{3}{2}.
 \]
 This yields a contradiction for sufficiently large $n$ since ${3}/{2}<2(1-\varepsilon)$.
\end{proof}

\begin{prop}\label{NOT Asymptotic l1}
	The space $T_{inc}^\xi$ is not Asymptotic $\ell_1$.
\end{prop}
\begin{proof}
	Suppose that $T_{inc}^\xi$ is $C$-Asymptotic $\ell_1$ and let $n\in\N$ be such that $n> C^2$. Since $T_{inc}^\xi$ is reflexive, we may assume that player (S) chooses tail subspaces (see \cite[Lemma 5.18]{AGLM}) throughout any winning strategy in the game $G(n,1,C)$. Let us assume the role of player (V) and let $Y_1$ be the tail subspace with which player (S) initiates the game. Then, as player $(V)$, we choose an element of the basis $e_{j_1}\in Y_1$, such that $|S|\ge n$ for every maximal segment $S$ of $\T_\xi$ with $\min S={j_1}$. Suppose that in the $k+1$ turn of the game, for $k<n$, player $(S)$ chooses the subspace $Y_{k+1}$. Then, again as player (V), we choose a vector $e_{j_{k+1}}\in Y_{k+1}$ with $j_{k+1}$ an immediate successor of $j_k$. Note that, in the final outcome of the game, we have chosen elements of the basis $e_{j_1},\ldots,e_{j_n}$ such that $\{j_1,\ldots,j_n\}$ is a segment of $\T_\xi$ and hence $\{e_{j_1},\ldots,e_{j_n}\}$ is isometric to the standard basis of $\ell_2^n$. We calculate
	\[
	\Big\| \frac{1}{n}\sum_{i=1}^n e_{j_i}\Big\|=	\Big\| \frac{1}{n}\sum_{i=1}^n e_{j_i}\Big\|_{G^\xi_2}=n^{-\frac{1}{2}}
	\]
	whereas, since $T_{inc}^\xi$ is $C$-Asymptotic $\ell_1$, we have that
	\[
	\frac{1}{C}\le\Big\| \frac{1}{n}\sum_{i=1}^n e_{j_i}\Big\|
	\]
and this is a contradiction.
	\end{proof}

\begin{rem}
For any $1<p<\infty$, we may replace the norming set $G^\xi_2$ with
\[
G^\xi_p=\bigg\{\sum_{i\in S}a_ie^*_i:\text{ S is a segment of }\T_\xi\text{ and }\sum_{i\in S}|a_i|^q\le1 \bigg\}
\]
where $p^{-1}+q^{-1}=1$, to obtain a reflexive Banach space admitting a uniformly unique $\ell_1$ joint spreading model, that contains a weakly null $\ell_p$-tree of height $\omega^\xi$ or a weakly null $c_0$-tree if we replace $G^\xi_2$ with $G=\{\pm e^*_i:i\in\N\}$.

\end{rem}

%--------------------------------------------------------------------------------- ell_1 basis --------------------------------------------------------------------------------------------------------

\section{A stronger separation of the two properties}

The spaces $T^\xi_{inc}$ constructed in the previous section, yield a separation between the properties of being an Asymptotic $\ell_1$ space and admitting a unique $\ell_1$ asymptotic model. It is easy however to see that these spaces contain subsequences of their bases that generate Asymptotic $\ell_1$ subspaces. For example, consider any subspace generated by a subsequence $(e_j)_{j\in M}$ of the basis of some $T^\xi_{inc}$, such that the elements of $M$ are pairwise incomparable in $\T_\xi$. In this section we show that, for any countable ordinal $\xi$, there is a reflexive Banach space $T^\xi_{ess\-inc}$ that admits a unique $\ell_1$ asymptotic model with respect to $\mathscr{F}_b(T^\xi_{ess\-inc})$ and any subsequence of its basis generates a non-Asymptotic $\ell_1$ subspace. To some extent, this family of spaces is the Maurey - Rosenthal \cite{MR} analogue of the two aforementioned properties.

Start by fixing a countable ordinal $\xi$ and let $(m_j)_{j\ge0}$, $(n_j)_{j\ge 0}$ be increasing sequences of natural numbers such that :
\begin{itemize}
	\item[(i)] $m_0=2$, $m_1=4$ and $m_j\ge m_{j-1}^2$ for every $j\ge 2$ and
	\item[(ii)] $n_0=1$, $n_1=6$ and $n_j>\log_2m^2_{j}+n_{j-1}$ for every $j\ge2$.
\end{itemize}
Let $\mathcal{Q}$ denote the collection of all finite sequences $((g_1,m_{j_1}),\ldots,(g_{k},m_{j_k}))$, where $g_i:\N\to\{-1,0,1\}$ has finite support and $j_i\in\N$ for $1\leq i\leq k$, and $m_{j_1}<\cdots<m_{j_k}$. Let $\sigma :\mathcal{Q}\to \{m_j:j\in\N\}$ be an injection  so that each sequence $((g_1,m_{j_1}),\ldots,(g_{k},m_{j_k}))$ is mapped to some $m_j$ with $m_{j_k}<m_j$.

\begin{dfn}\label{definition full tree support+weight}
	Let $\tT_\xi$ be the set of all finite sequences $((g_1,m_{j_1}),\ldots,(g_k,m_{j_k}))$ satisfying the following conditions.
	\begin{itemize}
		\item[(i)] $g_i:\N\to\{-1,0,1\}$ for $i=1,\ldots,k$ with $\supp g_1<\ldots<\supp g_k$.
		\item[(ii)] $\supp g_i\in \S_{n_{j_i}}$ for $i=1,\ldots,k$, where $n_{j_1}=n_1$ and $n_{j_1}<\ldots<n_{j_k}$.
		\item[(iii)] $m_{j_1}=m_1$ and $m_{j_i}=\sigma ((g_1,m_{j_1}),\ldots,(g_{i-1},m_{j_{i-1}}))$ for every $i=2,\ldots,k$.
		\item[(iv)] $\{\min \supp g_i:i=1,\ldots,k\}\in\S_\xi$.
	\end{itemize}
\end{dfn}

Note that item (iii) of the above definition implies that $\tT_\xi$, equipped with the partial order $\le_{\tT_\xi}$ where $\tt_1\le_{\tT_\xi}\tt_2$ if $\tt_1$ is an initial segment of $\tt_2$, is a tree. Moreover, it is easy to see that it is infinite-branching, and as follows from item (iv) and standard inductive arguments, it is also well founded and of height $\omega^\xi$. In particular, the above remain true if for an infinite subset of the naturals $M$ we additionally require that $\supp g_i\subset M$ for every $i=1,\ldots,k$, in Definition \ref{definition full tree support+weight}.

We may also identify $\tT_\xi$ as a closed subset, with respect to the pointwise convergence topology, of $\big\{ \{{\pm m_j^{-1}}\}_{j\in\N}\cup\{0\} \big\}^\N$ via the mapping
\[
((g_1,m_{j_1}),\ldots,(g_k,m_{j_k}))\mapsto {m^{-1}_{j_1}}\;g_1+\cdots+{m^{-1}_{j_k}}\;g_k.
\]
The fact that $\lim_jm_j^{-1}=0$ implies that $\big\{ \{{\pm m_j^{-1}}\}_{j\in\N}\cup\{0\} \big\}^\N$ is compact with respect to the pointwise convergence topology of $[-1,1]^\mathbb{N}$.

Observe that, as a consequence of item (iii), any $\tt=((g_1,m_{j_1}),\ldots,(g_k,m_{j_k}))$ in $\tT_\xi$ is uniquely determined by the pair $(g_k,m_{j_k})$, which we will denote by $(g_t,m_{j_t})$ or just by $t$ (i.e., $t = (g_t,m_{j_t})$. Taking advantage of this we may define $\T_\xi=\{(g_t,m_{j_t}):\tt\in\tT_\xi \}$, which is in bijection with $\tT_\xi$ via the mapping $t=(g_t,m_{j_t})\mapsto \tt$. Note that $\le_{\tT_\xi}$ induces a natural order, denoted by $\le_{\T_\xi}$, on $\T_\xi$, where $(g_{t_1},m_{j_{t_1}})\le_{\T_\xi}(g_{t_2},m_{j_{t_2}})$ if $\tt_1\le_\tT \tt_2$.  Clearly, the tree $(\T_\xi,\le_{\T_\xi})$ is isomorphic to $(\tT_\xi,\le_{\tT_\xi})$ via the mapping $t=(g_t,m_{j_t})\mapsto \tt$.

\begin{dfn}
	Let $\tW_\xi$ be the set of all finite sequences $(m_{j_1},m_{j_2},\ldots,m_{j_k})$\linebreak for which there exist $g_1,\ldots,g_k:\N\to\{-1,0,1\}$  with $((g_1,m_{j_1}),\ldots,(g_k,m_{j_k}))\in\tT_\xi$.
\end{dfn}

The initial segment order $\le_{\tW_\xi}$ is a partial order on $\tW_\xi$ and is in fact naturally induced by the order $\le_{\tT_\xi}$. Moreover, it is easy to verify that $(\tW_\xi,\le_{\tW_\xi})$ is a well founded infinite-branching tree of height $\omega^\xi$. It is also isomorphic to the tree $(\W_\xi,\le_{\W_\xi})$, where $\W_\xi=\{m_{j_t}:\tt\in\tT_\xi \}$ and $m_{j_{t_1}}\le_{\W_\xi} m_{j_{t_2}}$ if $\tt_1\le_{\tT_\xi}\tt_2$. This correspondence between $\tW_\xi$ and $\W_\xi$ is identical to that of $\tT_\xi$ and $\T_\xi$.

\begin{rem}
	\begin{itemize}
		\item[(i)] If $m_{j_{t_1}}$, $m_{j_{t_2}}$ are incomparable nodes  in $\W_\xi$, then for every $g_1,g_2:\N\to\{-1,0,1\}$ such that $(g_1,m_{j_{t_1}})$ and $(g_2,m_{j_{t_2}})$ are in $\T_\xi$, these are also incomparable.
		\item[(ii)] 		Note that there exist nodes $\tt_1$ and $\tt_2$ which are incomparable in $\tT_\xi$, whereas $m_{j_{t_1}}$ and $m_{j_{t_2}}$ are comparable in $\W_\xi$. To see this, consider any node\linebreak $\tt=((g_1,m_{j_1}),\ldots,(g_k,m_{j_k}))$ in $\tT_\xi$ with $k>1$ and, for each $i=1,\ldots,k-1$, let $h_i:\N\to\{-1,0,1\}$ be such that $h_i\neq g_i$ and $t_i=(h_i,m_{j_i})$ is in $\T_\xi$. Then, item (iii) of Definition \ref{definition full tree support+weight} implies that the nodes $\tt_i$ and $\tt$ are incomparable whereas $m_{j_{t_i}}$ and $m_{j_{t}}$ are comparable for every $i=1,\ldots,k-1$, since $\tt\in\tT_\xi$.
	\end{itemize}
\end{rem}

\begin{dfn}
	We say that a subset $X$ of $\T_\xi$ is essentially incomparable if whenever $(g_{t_1},m_{j_{t_1}})$,  $(g_{t_2},m_{j_{t_2}})$ are in $X$ with $m_{j_{t_1}}<_{\W_\xi} m_{j_{t_2}}$ and $g:\N\to\{-1,0,1\}$ is the unique sequence such that $(g,m_{j_{t_1}})\le _{\T_\xi} (g_{t_2},m_{j_{t_2}})$, then $\supp g<\supp g_{t_1}$.
\end{dfn}

\begin{rem}
Let $X$ be an essentially incomparable subset of $\T_\xi$ and $h_t:\N\to\{-1,0,1\}$ with $\supp h_t\subset\supp g_t$ for every $t\in X$. Then $\{(h_t,m_{j_t}):t\in X\}$ is also an essentially incomparable subset of $\T_\xi$.
\end{rem}

The following lemma is an extension of Proposition \ref{proposition incomparable} and is the main ingredient of the proof that the space $T^\xi_{ess\- inc}$ admits a uniformly unique joint spreading model.

\begin{lem}\label{combinatorial lemma essentially incomparable}
	Let $(\mu_i)_i$ be a sequence of positive regular measures on $\tT_\xi$ with finite supports and let $C>0$ be such that $\mu_i(\tT_\xi)<C$ for all $i\in\N$. Assume that the sets $\cup\{\supp g_t:\tt\in\supp\mu_i\}$, $i\in\N$, are disjoint. Then, for every $\varepsilon>0$, there exists an $M\in[\N]^\infty$ and $G^1_i$, $G^2_i$ subsets of $\tT_\xi$ for each $i\in M$, such that
	\begin{itemize}
		\item[(i)] $G^1_i$, $G^2_i$ are disjoint subsets of $\supp\mu_i$ for every $i\in M$,
		\item[(ii)] $\mu_i(\tT_\xi\setminus G^1_i\cup G^2_i)<\varepsilon$ for every $i\in M$,
		\item[(iii)] $\{t\in\T_\xi:\tt\in \cup_{i\in M}G^1_i\}$ is essentially incomparable and
		\item[(iv)] for every $i_1\neq i_2$ in $M$, every $\tt_1\in G^2_{i_1}$ and $\tt_2\in G^2_{i_2}$, the nodes $m_{j_{t_1}}$ and $m_{j_{t_2}}$ are incomparable in $\W_\xi$.
	\end{itemize}
\end{lem}

Before we are able to prove this Lemma it is necessary to introduce the notion of successor limits of measures. We find this limit notion to be of independent interest and therefore we use broader terminology to define it and prove its properties.

\begin{notn}
Let $\mathcal{T}$ be a countably branching well founded tree. For each $t\in\mathcal{T}$ we denote $\mathrm{succ}_{\mathcal{T}}(t)$ the set of immediate successors of $t$. In particular, if $t$ is maximal then $\mathrm{succ}_{\mathcal{T}}(t)$ is empty. For $t\in\mathcal{T}$ we denote $V_t = \{s\in\mathcal{T}:t\leq s\}$. We view $\mathcal{T}$ as topological space with the topology generated by the sets $V_t$ and $\mathcal{T}\setminus V_t$, $t\in\mathcal{T}$. This is a compact metric topology for which the sets of the form $V_t\setminus(\cup_{s\in F}V_s)$, $t\in\mathcal{T}$ and $F\subset\mathrm{succ}_\mathcal{T}(t)$ finite, form a base of clopen sets. We denote by $\mathcal{M}_+(\mathcal{T})$ the cone of all bounded positive measures $\mu:\mathcal{P}(\mathcal{T})\to[0,+\infty)$. For $\mu\in\mathcal{M}_+(\mathcal{T})$ we define the support of $\mu$ to be the set $\mathrm{supp}(\mu) = \{t\in\mathcal{T}:\mu(\{t\})>0\}$.  A set $A$ in $\mathcal{M}_+(\mathcal{T})$ is called bounded if $\sup_{\mu\in\mathcal{A}}\mu(\mathcal{T})<\infty$.
\end{notn}

Recall that a sequence $(\mu_i)$ in $\mathcal{M}_+(\mathcal{T})$ converges in the $w^*$-topology to a $\mu\in\mathcal{M}_+(\mathcal{T})$ if and only if for all clopen sets $V\subset \mathcal{T}$ we have $\lim_i\mu_i(V) = \mu(V)$ if and only if for all $t\in\mathcal{T}$ we have $\lim_i\mu_i(V_t) = \mu(V_t)$.

\begin{dfn}
Let $\mathcal{T}$ be a countably branching well founded tree, $(\mu_i)$ be a disjointly supported sequence in $\mathcal{M}_+(\mathcal{T})$ and $\nu\in\mathcal{M}_+(\mathcal{T})$. We say that $\nu$  is the successor-determined limit of $(\mu_i)$ if for all $t\in\mathcal{T}$ we have $\mu(\{t\}) = \lim_i\mu_i(\mathrm{succ}_\mathcal{T}(t))$. In this case we write $\nu = \mathrm{succ}\-\!\lim_i\mu_i$.
\end{dfn}

%\begin{rem}
%Any bounded sequence $(\mu_i)$ of disjointly supported positive measures has a subsequence so that for all $t\in\mathcal{T}$ the limit $\nu(\{t\}) = \lim_i\mu_i(\mathrm{succ}_\mathcal{T}(t))$ exists.
%\end{rem}

\begin{rem}
It is possible for a disjointly supported and bounded sequence $(\mu_i)\in\mathcal{M}_+(\mathcal{T})$ to satisfy $w^*\-\lim_i\mu_i\neq \mathrm{succ}\-\!\lim_i\mu_i$. Take for example $\mathcal{T} = [\mathbb{N}]^{\leq 2}$ (all subsets of $\mathbb{N}$ with at most two elements with the partial order of initial segments). Define $\mu_i = \delta_{\{i,i\}}$. Then, $w^*\-\lim_i\mu_i = \delta_\emptyset$ whereas $\mathrm{succ}\-\!\lim_i\mu_i = 0$.
\end{rem}

Although these limits are not the necessarily the same, there is an explicit formula relating $\mathrm{succ}\-\!\lim_i\mu_i$ to $w^*\-\lim_i\mu_i$.

\begin{lem}
\label{formula w-succ}
Let $\mathcal{T}$ be a countable well founded tree, $(\mu_i)$ be a bounded and disjointly supported sequence in $\mathcal{M}_+(\mathcal{T})$ so that $w^*\-\lim_i\mu_i = \mu$ exists and for all $t\in\mathcal{T}$  the limit $\nu(\{t\}) = \lim_i\mu_i(\mathrm{succ}_\mathcal{T}(t)) $  exists as well. Then, for every $t\in\mathcal{T}$ and enumeration $(t_j)$ of $\mathrm{succ}_\mathcal{T}(t)$ we have
\begin{equation}
\label{formula w-succ formula}
\mu(\{t\}) = \nu(\{t\}) + \lim_j\lim_i\mu_i\Big(\cup_{k\geq j}(V_{t_k}\setminus\{t_{k}\})\Big).
\end{equation}
In particular, $\mu(\{t\}) = \nu(\{t\})$ if and only if the double limit in \eqref{formula w-succ formula} is zero.
\end{lem}

\begin{proof}
For $j\in\mathbb{N}$ we have $\{t\}\cup(\cup_{k\geq j}V_{t_k}) = V_t\setminus(\cup_{k<j}V_{t_k})$ which is clopen and thus
\begin{equation}
\label{thankfully its clopen}
\lim_i\mu_i\big(\{t\}\cup(\cup_{k\geq j}V_{t_k})\big) = \mu\big(\{t\}\cup(\cup_{k\geq j}V_{t_k})\big).
\end{equation}
Because $(\mu_i)$ is disjointly supported we observe that  for all $j\in\N$
\begin{equation}
\label{formula w-succ eq1}
\lim_i\mu_i(\{t_{k}:k\geq j\}) = \lim_i\mu_i(\mathrm{succ}_\mathcal{T}(t)) = \nu(\{t\}).
\end{equation}
We calculate
\begin{equation*}
\begin{split}
\mu(\{t\}) &= \lim_{j\to\infty}\mu\Big(\{t\}\cup(\cup_{k\geq j}{V}_{t_j})\Big) \stackrel{\eqref{thankfully its clopen}}{=} \lim_j\lim_i\mu_i\Big(\{t\}\cup(\cup_{k\geq j}V_{t_j})\Big)\\
&= \lim_j\lim_i\mu_i(\cup_{k\geq j}V_{t_j}) = \lim_j\lim_i\mu_i(\{t_{k}:k\geq j\}\cup(\cup_{k\geq j}(V_{t_k}\setminus\{t_{k}\})))\\
&= \lim_j\lim_i\mu_i(\{t_{k}:k\geq j\}) + \lim_j\lim_i\mu_i\Big(\cup_{k\geq j}(V_{t_k}\setminus\{t_{k}\})\Big).
\end{split}
\end{equation*}
Thus, \eqref{formula w-succ eq1} yields the conclusion.
\end{proof}

\begin{cor}
Let $\mathcal{T}$ be a countable well founded tree and $(\mu_i)$ be a bounded and disjointly supported sequence in $\mathcal{M}_+(\mathcal{T})$. Then, there exist a subsequence $(\mu_{i_n})$ of $(\mu_i)$ and $\nu\in\mathcal{M}_+(\mathcal{T})$ with $\nu = \mathrm{succ}\-\!\lim_n\mu_{i_n}$.
\end{cor}

\begin{proof}
By passing to a subsequence, $\mu = w^*\-\lim_i\mu_i$ exists and for all $t\in\mathcal{T}$  the limit $\nu(\{t\}) = \lim_i\mu_i(\mathrm{succ}_\mathcal{T}(t)) $  exists as well. By \eqref{formula w-succ formula} for all $t\in\mathcal{T}$ we have $\nu(\{t\})\leq \mu(\{t\})$. Thus $\sum_{t\in\mathcal{T}}\nu(\{t\}) \leq \mu(\mathcal{T})$, i.e., $\nu$ defines a bounded positive measure
\end{proof}

\begin{lem}
\label{splitting lemma}
Let $\mathcal{T}$ be a countable well founded tree and $(\mu_i)$ be a bounded and disjointly supported sequence in $\mathcal{M}_+(\mathcal{T})$ so that $\mathrm{succ}\-\!\lim_i\mu_i = \nu$  exists. Then, there exist an infinite $L\subset\mathbb{N}$ and partitions $A_i$, $B_i$ of $\mathrm{supp}(\mu_i)$, $i\in L$, so that the following are satisfied.
\begin{itemize}

\item[(i)] If for all $i\in L$ we define the measure $\mu_i^1$ given by  $\mu_i^1(C)= \mu_i(C\cap A_i)$, then $\nu = w^*\-\lim_{i\in L}\mu_i^1 = \mathrm{succ}\-\!\lim_{i\in L}\mu_i^1$.

\item[(ii)] If for all $i\in L$ we define the measure $\mu_i^2$ given by  $\mu_i^2(C)= \mu_i(C\cap B_i)$ then for all $t\in\mathcal{T}$ the sequence $(\mu_i^2(\mathrm{succ}_\mathcal{T}(t)))_i$ is eventually zero. In particular, $\mathrm{succ}\-\!\lim_{i\in L}\mu_i^2 = 0$.

\end{itemize}
\end{lem}

\begin{proof}
Enumerate $\mathcal{T} = \{s_n:n\in\mathbb{N}\}$ and assume, passing if necessary to a subsequence, that for all $n\in\mathbb{N}$ and $i>  n$ we have
\begin{equation}
\label{sufficiently close}
|\mu_{i}(\mathrm{succ}_\mathcal{T}(s_n)) - \nu(s_n)| < \frac{1}{2^n}.
\end{equation}
Let us point out that for $m\neq n$ the sets $\mathrm{succ}_\mathcal{T}(s_m)$ and $\mathrm{succ}_\mathcal{T}(s_n)$ are disjoint and $\cup_n\mathrm{succ}_\mathcal{T}(s_n) = \mathcal{T}\setminus\{t_0\}$, where $t_0$ denotes the root of the tree $\mathcal{T}$.  We may, and will, assume that for all $i\in\mathbb{N}$, $t_0\not\in\mathrm{supp}(\mu_i)$. Define for each $i\in \mathbb{N}$ the sets
\begin{equation*}
A_i = \mathrm{supp}(\mu_i)\cap\Big(\cup_{n=1}^i\mathrm{succ}_\mathcal{T}(s_n)\Big)\text{ and } B_i = \mathrm{supp}(\mu_i)\cap\Big(\cup_{n=i+1}^\infty\mathrm{succ}_\mathcal{T}(s_n)\Big).
\end{equation*}
We point out that for all $i\in\mathbb{N}$, $A_i$, $B_i$ forms a partition of $\mathrm{supp}(\mu_i)$ and we will show that it has the desired properties.

Statement (ii) follows directly from the fact that for every $t\in\mathcal{T}$ the sequence of sets $(B_i\cap\mathrm{succ}_\mathcal{T}(t))_i$ is eventually empty. To show that (i) holds we fix $t\in\mathcal{T}$ and let $(t_j)$ be an enumeration of $\mathrm{succ}_\mathcal{T}(t)$. Define $L_j = \cup_{k=j}^\infty\{n\in\mathbb{N}:t_k\leq s_n\}$, for each $j\in\mathbb{N}$, and observe that $\cap_j L_j = \emptyset$. Also observe that for all $j\in\mathbb{N}$ we have $\cup_{k\geq j}(V_{t_k}\setminus\{t_k\}) = \cup_{n\in L_j}\mathrm{succ}_\mathcal{T}(s_n)$. Therefore we have
\begin{equation*}
\begin{split}
\mu_i&\Big(A_i \cap\big(\cup_{k\geq j}(V_{t_k}\setminus\{t_k\})\big) \Big) = \mu_i\Big(\big(\cup_{n=1}^i\mathrm{succ}_\mathcal{T}(s_n)\big)\cap\big(\cup_{n\in L_j}\mathrm{succ}_\mathcal{T}(s_n)\big)\Big)\\
&= \mu_i\Big(\cup_{n\in L_j\cap[1,i]}\mathrm{succ}_\mathcal{T}(s_n)\Big) = \sum_{n\in L_j\cap[1,i]} \mu_i\big(\mathrm{succ}_\mathcal{T}(s_n)\big)\\
&\stackrel{\eqref{sufficiently close}}{\leq} \sum_{n\in L_j\cap[1,i]}\nu(s_n) + \sum_{n\in L_j\cap[1,i]}\frac{1}{2^n} \leq \nu(\{s_n :n\in L_j\}) + 2^{-\min(L_j)+1}\\
&= \nu(\cup_{k\geq j}V_{t_k}) + 2^{-\min(L_j)+1}.
\end{split}
\end{equation*}
Therefore, $\lim_j\sup_i \mu_i\Big(A_i \cap\big(\cup_{k\geq j}(V_{t_k}\setminus\{t_k\})\big) \Big) = 0$ and by Lemma \ref{formula w-succ}, (i) is satisfied.
\end{proof}

	\begin{proof}[Proof of Lemma \ref{combinatorial lemma essentially incomparable}]
Apply Lemma \ref{splitting lemma} so that, by passing to a subsequence of $(\mu_i)$, there are, for each $i\in\mathbb{N}$, partitions $A_i$, $B_i$ of $\mathrm{supp}(\mu_i)$ so that the conclusion of that Lemma it satisfied. Define, for each $i\in\mathbb{N}$, the measures $\mu_i^1$, $\mu_i^2$ given by $\mu_i^1(C) = \mu_i(A_i\cap C)$  and $\mu_i^2(C) = \mu_i(B_i\cap C)$. Let $\nu = w^*\-\lim_i\mu^1_i = \mathrm{succ}\-\!\lim_i\mu^1_i$. Pick a finite subset $F$ of $\tT_\xi$ so that $\nu(\tT_\xi\setminus F) <\varepsilon/2$. Then, because $\nu = w^*\-\lim_i\mu^1_i$ we have $\lim_i\mu_i^1(\tT_\xi) = \nu(\tT_\xi)$ and because $\nu =  \mathrm{succ}\-\!\lim_i\mu^1_i$
\begin{equation*}
\lim_i\Big|\mu^1_i(\tT_\xi) - \mu^1_i(\cup_{\tilde t\in F}\mathrm{succ}(\tilde t))\Big| = \Big|\nu(\tT_\xi)  - \lim_i\sum_{\tilde t\in F}\mu_i^1(\mathrm{succ}(\tilde t))\Big| = \nu(\tT_\xi\setminus F)<\frac{\varepsilon}{2}.
\end{equation*}
We can find $i_0\in\mathbb{N}$ so that for all $i\geq i_0$ we have
\begin{equation}
\label{successor is close enough}
\Big|\mu_i(A_i) - \mu_i\Big(A_i\cap\big(\cup_{\tilde t\in F}\mathrm{succ}(\tilde t)\big)\Big)\Big| = \Big|\mu^1_i(\tT_\xi) - \mu^1_i(\cup_{\tilde t\in F}\mathrm{succ}(\tilde t))\Big| <\frac \varepsilon 2.
\end{equation}
We may, using the fact that  the sets $\cup\{\supp g_t:\tt\in\supp\mu_i\}$ for $i\in\N$ are disjoint, find $j_0\geq i_0\in\N$ such that
\begin{equation}
\label{avoid all supports}
\cup_{\tilde s\in F}\supp g_{s} < \supp g_t\text{ for every }\tilde t\in\cup_{i\geq j_0}\mathrm{supp}(\mu_i^1).
\end{equation}
We define $G_i^1 = A_i\cap(\cup_{\tt\in F}\mathrm{succ}(\tt))$, $i\geq j_0$. By \eqref{successor is close enough} we have that for all $i\geq j_0$, $|\mu_i(A_i) - \mu_i(G_i^1)|<\varepsilon/2$. Additionally, $\{t\in\T_\xi:\tt\in\cup_{i\geq j_0}G_1^i\}$ is essentially incomparable. Indeed, let $\ts_1,\ts_2\in \cup_{i\geq j_0}G^1_i$ with $m_{j_{s_1}}<_{\W} m_{j_{s_2}}$ and $(h,m_{j_{s_1}})\in\T_\xi$ be such that and $(h,m_{j_{s_1}})\le_{\T_\xi} s_2$. Then \eqref{avoid all supports} implies that $\supp h<\supp g_{s_1}$.

For the remaining part of the proof, since for all $i\in\N$ the set $B_i = \supp\mu_i^2$ is finite (as a subset of the finite support of $\mu_i$) and for each $\tt\in\tT_\xi$ the sequence $(\mu_i^2(\mathrm{succ}(\tt)))_i$ is eventually zero, we may pass to a subsequence so that for all $i<j$ we have $\{m_{j_t}:\tt\in \supp\mu_i^2\}\cap \{m_{j_t}:\tt\in\supp\mu_j^2\} = \emptyset$. We can therefore define the bounded sequence of  disointly supported measures $(\nu_i)$ on $\widetilde{\mathcal{W}}_\xi$ with $\nu_i(\{(w_1,\ldots,w_k)\}) = \mu_i^2(\{\tilde t\in\tT_\xi: m_{j_t} = w_k\})$. Hence, applying Proposition \ref{proposition incomparable} and passing to a subsequence, we obtain a subset $E_i$ of $\supp\nu_i$ such that $\nu_i(\tW_\xi\setminus E_i)<\varepsilon/2$ and the sets $E_i$, $i\in\N$, are pairwise incomparable. It is easy to verify that $G^2_i=\{\tt\in B_i:m_{j_t}\in 2_i \}$, $i\in\N$, are pairwise incomparable and $|\mu_i(B_i) - \mu_i(G_i^2)| = \mu^2_i(\tT_\xi\setminus G^2_i)<\varepsilon/2$ for every $i\in\N$.
\end{proof}

We now define the space $T^\xi_{ess\-inc}$ in a similar way to $T^\xi_{inc}$, that is, using the notion of the Tsirelson extension $W_G$ of a ground set $G$.

\begin{dfn}
	Define the following norming sets on $c_{00}(\N)$.
	\[
	G_0=\big\{\pm e^*_n:n\in\N \big\}
	\]
	\[
	G_1=\Big\{\frac{1}{m_j}\sum_{n\in \N}g(n)e^*_n:j\in\N\;\text{ and }g:\N\to\{-1,0,1\}\text{ with }\supp g\in \S_{n_j}			 \Big\}.
	\]
	For each $f=m_j^{-1}\sum_{n\in \N}g(n)e^*_n$ in $G_1$, set $t_f=(g,m_j)$. Moreover, if $G=G_1\cup G_0$ and $f$ is in $W_G$ with a tree analysis $(f_\alpha)_{\alpha\in \mathcal{A}}$, define
	\[
	\M^1_f=\{\alpha:\alpha\text{ is a maximal node of }\mathcal{A}\text{ and }f_a\in G_1 \}.
	\]
	Let $W$ be the subset of $W_G$ containing all functionals $f$ such that $\{t_{f_\alpha}:\alpha\in\M^1_f \}$ is an essentially incomparable subset of $\T_\xi$. Denote by $T^\xi_{ess\-inc}$ the completion of $c_{00}(\N)$ with respect to the norm $\|\cdot\|_W$ induced by $W$.
\end{dfn}

\begin{rem}
	\begin{itemize}
		\item[(i)] 		The standard basis $(e_j)_j$ of $c_{00}(\N)$ forms a $1$-unconditional basis for the space $T^\xi_{ess\-inc}$ and it is also boundedly complete since $T^\xi_{ess\-inc}$ admits a uniformly unique $\ell_1$ spreading model as shown in Proposition \ref{first space unique}.			
		\item[(ii)]  If $f\in W_G$ with a tree analysis $(f_\alpha)_{\alpha\in \mathcal{A}}$  and $m_{j_{t_{f_\alpha}}}$, for $\alpha\in\M^1_f$, are pairwise incomparable nodes in $\W_\xi$ , then $f\in W$.
		\item[(iii)] The norming set of $T^\xi_{ess\- inc}$ contains the norming set of Tsirelson's original space, i.e., the Tsirelson extension of $G_0$.
	\end{itemize}
	
\end{rem}

\begin{prop}\label{first space unique}
	The space $T^\xi_{ess\-inc}$ admits $\ell_1$ a uniformly unique joint spreading model with respect to $\mathscr{F}_b(T^\xi_{ess\-inc})$.
\end{prop}
\begin{proof}
	Let $(x^i_j)_j$, $1\le i \le l$, be an array of normalized block sequences in $T^\xi_{ess\-inc}$ and fix $\varepsilon>0$. Passing to a subsequence, we may assume that $\supp x^{i_1}_j<\supp x^{i_2}_{j+1}$ for all $i_1, i_2=1,\ldots,l$ and $j\in\N$. For each $i=1,\ldots,l$ and $j\in\N$, pick a functional $f^i_j=\sum_{\alpha\in \M^i_j}f^i_{j,\alpha}/2^{k^i_{j,\alpha}}$ in $W$ with $f^i_j(x^i_j)\ge 1-\varepsilon$ and $f^i_{j,\alpha}(x^i_j)>0$ for every $\alpha\in\M^i_j$, where $\mathcal{M}^i_j$ denotes the set of all maximal nodes of a fixed tree analysis of $f^i_j$. Moreover, for each $\alpha\in \M^1_{f^i_j}=\{\alpha\in\M^i_j:f^i_{j,\alpha}\in G_1\}$, define $t^i_{j,\alpha}=t_{f^i_{j,\alpha}}$ and, for each $j\in\N$, the measure $\mu_j$ as follows:
	\[
	\mu_j = \sum_i\sum_{\alpha\in \M^1_{f^i_j}}\frac{f^i_{j,\alpha}(x^i_j)}{2^{k^i_{j,\alpha}}}\delta_{\tt^i_{j,\alpha}}.
	\]
	Passing to a subsequence assume that $\lim_j\mu_j(\tT_\xi)=c$. If $c=0$, then we may assume that $f^i_{j,\alpha}\in G_0$ for every $i=1,\ldots,l$, $j\in\N$ and $\alpha\in\M^i_j$, in which case the desired result is immediate. Hence, if $c>0$, applying Lemma \ref{combinatorial lemma essentially incomparable} and passing to a subsequence, we obtain $(G^1_j)_j$, $(G^2_j)_j$ satisfying items (i) - (iv) with $\mu_j(\tT_\xi\setminus G^1_j\cup G^2_j)<1/8$. Then, for each pair $(i,j)$, set
	\[
	\mathcal{M}^1_{i,j}=\{\alpha\in\M^i_{f^i_j}:t^i_{j,\alpha}\in G^1_j  \}\quad\text{ and }\quad\mathcal{M}^2_{i,j}=\mathcal{M}_{i}^j\setminus \mathcal{M}^1_{i,j}
	\]
	and
	\[
	f^k_{i,j}=\sum_{\alpha\in \mathcal{M}^k_{i,j}}f^i_{j,\alpha}/2^{k^i_{j,\alpha}},\quad k=1,2.
	\]
	Note in particular that, for every pair $(i,j)$, the fact that $\mu_j(\tT_\xi\setminus G^1_j\cup G^2_j)<1/8$ implies that $|f^i_j(x^i_j)-(f^1_{i,j}(x^i_j)+f^2_{i,j}(x^i_j))|<1/8$ and hence that there exists $k=1,2$ such that $f^k_{i,j}(x^i_j)\ge (7-\varepsilon)/16$. Set
	\[
	A_k=\{(i,j) :f^{k}_{i,j}(x^i_j)\ge (7-8\varepsilon)/16 \},\quad k=1,2.
	\]
	Let $n\in\N$, $\{\lambda_{ij}\}_{i=1,j=1}^{l,n}\subset[-1,1]$ with $\sum_{i,j}|\lambda_{ij}|=1$ and $s=(s_i)_{i=1}^l\in S\- Plm_l([\N]^k)$ with $ln\le\min\supp x^1_{s_1(1)}$. Then let $k=1,2$ be such that $\sum_{(i,s_i(j))\in A_k}|\lambda_{ij}|\ge 1/2$ and observe that $f=1/2\sum_{(i,s_i(j))\in A_k}f^k_{i,s_i(j)}$ is in $W$. Hence, we calculate
	\[
	\bigg\|\sum_{i=1}^l\sum_{j=1}^n|\lambda_{ij}|x^i_{s_i(j)} \bigg\|\ge f\bigg(\sum_{i=1}^l\sum_{j=1}^n|\lambda_{ij}|x^i_{s_i(j)}\bigg)=\frac{1}{2}\sum_{(i,s_i(j))\in A_k}|\lambda_{ij}|f^k_{i,s_i(j)}\big(x^i_{s_i(j)}\big)\ge\frac{7-8\varepsilon}{32}
	\]	
	and due to unconditionality this yields that
	\[
	\bigg\|\sum_{i=1}^l\sum_{j=1}^n\lambda_{ij}x^i_{s_i(j)}\bigg\|\ge\frac{7-8\varepsilon}{32}.
	\]
\end{proof}

It remains to show that for every $M\in[\N]^\infty$, the space $T^\xi_{ess\-inc}$  contains a $c_0$-tree of height $\omega^\xi$ supported by $(e_j)_{j\in M}$. To this end, let us recall the following definition.

\begin{dfn}
	Let $n\in\N$ and $\varepsilon>0$. We say that a convex combination $x=\sum_{i\in \Delta}\lambda_ie_i$ in $c_{00}(\N)$ is an $(n,\varepsilon)$-special convex combination if
	\begin{itemize}
		\item[(i)] $\Delta\in \S_n$ and
		\item[(ii)] $\sum_{i\in \Delta'}\lambda_i<\varepsilon$ for every $\Delta'\in \S_m$ with $m<n$. 		
	\end{itemize}	
\end{dfn}

The main ingredient in the proof of the following proposition is the notion of repeated averages, first defined by Argyros, Mercourakis, and Tsarpalias.    in \cite{AMT}. We refer the reader to \cite[Chapter 2]{AT} for further details.

\begin{prop}
	For every $n\in\N$ and $\varepsilon>0$, there is a $k\in\N$ such that, for every maximal subset $F$ of $\S_n$ with $k<F$, there exists an $(n,\varepsilon)$-special convex combination $x$ in $c_{00}(\N)$ with $\supp x=F$.
\end{prop}

For a functional $f$ in $W$ with tree analysis $(f_\alpha)_{\alpha\in \mathcal{A}}$, we define the height of $f$, denoted by $h(f)$, as the maximum of $|a|$ over all maximal nodes $\alpha\in\mathcal{A}$. Moreover, if $f=m^{-1}_j\sum_{n\in\N}g(n)e^*_n$ is in $W$, we say that $f$ is a weighted functional and define the weight of $f$ as $w(f)=m_j$.

\begin{lem}\label{admissibility height of tree}
	Let $j\in \N$ and $f$ be a functional in $W$ with a tree analysis $(f_\alpha)_{\alpha\in \mathcal{A}}$ such that $w(f_\alpha)<m_j$ for every $\alpha\in\mathcal{M}^1_f$. Then $\supp f\in \S_{k}$, where $k\le n_{j-1}+h(f)$.
\end{lem}
\begin{proof}
	For each $\alpha\in\mathcal{A}$, let $k_\alpha\in\N$ be such that $\supp f_\alpha\in\S_{k_\alpha}$. Then, since $w(f_\alpha)<m_j$, we have that $k_\alpha\le n_{j-1}$ for every $\alpha\in\M^1_f$. Note then that, as follows from the definition of $W_G$, this implies that $k_\alpha\le n_{j-1}+1$ for every $\alpha\in\mathcal{A}$ with $|\alpha|=h(f)-1$. In particular, a finite induction argument yields that $k_\alpha\le n_{j-1}+i$ whenever $|\alpha|=h(f)-i$ and this proves the desired result.
\end{proof}

\begin{prop}\label{special convex combination seminormalized}
	Let $j\in\N$ and $x=\sum_{i\in\Delta}\lambda_ie_i$ be an $(n_j,m^{-2}_j)$-special convex combination, then
	\[
	\frac{1}{m_j}\le\|x\|_W\le\frac{1}{m_j}+\frac{1}{m^2_j}.
	\]
\end{prop}
\begin{proof}
	Pick an $f$ in $W$ and define  $\Delta_1=\{i\in\Delta:|f(e_i)|>m_j^{-1} \}$ and $\Delta_2 =\Delta\setminus \Delta_1$. Consider the tree analysis $(f^1_\alpha)_{\alpha\in\mathcal{A}}$ of $f_1=f|_{\Delta_1}$ and note that $w(f^1_\alpha)<m_j$ for every $\alpha\in\mathcal{M}^1_{f_{1}}$. Indeed, if $w(f^1_\alpha)=m_{j'}\ge m_j$ for some $\alpha$, then for any $i\in \supp f^1_\alpha$ we have that $|f(e_i)|\le m^{-1}_{j'}$ and this is a contradiction. Moreover, the fact that $|f_1(e_i)|>m^{-1}_j$ for every $i\in \Delta_1=\supp f_1$ implies that $h(f_1)<\log_2m_{j-1}$ and hence the previous proposition yields that $\supp f_1\in\S_l$, where $l\le\log_2m_j+n_{j-1}<n_j$. Therefore, since $x=\sum_{i\in\Delta}\lambda_ie_i$ is an $(n_j,m^{-2}_j)$\-special convex combination, we have that
	\[
	|f|_{\Delta_1}(\sum_{i\in\Delta}\lambda_ie_i)|\le\sum_{i\in \Delta_1}\lambda_i<\frac{1}{m^{2}_j}.
	\]
	We also calculate
	\[
	|f|_{\Delta_2}(\sum_{i\in\Delta}\lambda_ie_i)|\le \frac{1}{m_j}\sum_{i\in \Delta_2}\lambda_i\le \frac{1}{m_j}
	\]
	and conclude that $\|x\|_W\le m^{-1}_j+m^{-2}_j$. For the remaining part notice that the functional $f=m_j^{-1}\sum_{i\in \Delta}e^*_i$ is in $W$.
\end{proof}

\begin{prop}\label{special convex combination upper bound}
	Let $j\in\N$ and $x=\sum_{i\in\Delta}\lambda_ie_i$ be an $(n_j,m^{-2}_j)$-special convex combination. Then $|f(x)|<2m_j^{-2}$, for every $f\in W$ with a tree analysis $(f_\alpha)_{\alpha\in\mathcal{A}}$ such that $w(f_\alpha)\neq m_j$ for all $\alpha\in \M^1_f$.
\end{prop}
\begin{proof}
	Define $\Delta_1=\{i\in\Delta:|f(e_i)|>m^{-2}_{j} \}$ and $\Delta_2=\Delta\setminus \Delta_1$ and let $(f^1_\alpha)_{\alpha\in\mathcal{A}_1}$ be the tree analysis of  $f_1=f|_{\Delta_1}$. Similar arguments as in the previous proof yield that $w(f^1_\alpha)<m^2_j<m_{j+1}$ and hence $w(f^1_\alpha)<m_j$ for all $\alpha\in\M^1_{f_1}$, since $w(f_\alpha)\neq m_j$ for all $\alpha\in M^1_f$. Moreover, since $|f_1(e_i)|>m^2_j$ for all $i\in \supp f_1$, we have that $h(f_1)<\log_2m_j^{2}$ and therefore Proposition \ref{admissibility height of tree} yields that $\Delta_1=\supp f_1\in\S_l$ with $l\le \log_2m_j^{2}+n_{j-1}<n_j$. The fact that $x=\sum_{i\in\Delta}\lambda_ie_i$ is an $(n_j,m^{-2}_j)$\-special convex combination implies that
	\[
	|f_1(\sum_{i\in\Delta}\lambda_ie_i)|\le\sum_{i\in \Delta_1}\lambda_i<\frac{1}{m_j^2}.
	\]	
	We also calculate
	\[
	|f|_{\Delta_2}(\sum_{i\in\Delta}\lambda_ie_i)|\le\frac{1}{m^2_j}\sum_{i\in\Delta_2}\lambda_i\le\frac{1}{m^2_j}
	\]
	and this completes the proof.
\end{proof}

Let $M$ be an infinite subset of the naturals and consider the collection $T_\xi(M)$ of all finite sequences $(x_1,\ldots,x_k)$ of vectors in $c_{00}(\N)$ such that
\begin{itemize}
	\item[(i)] $x_l=m_{j_l}\sum_{i\in \Delta_l}\lambda_ie_i$, where $\sum_{i\in \Delta_l}\lambda_ie_i$ is an $(n_{j_l},m^{-2}_{j_l})$\-special convex combination for every $l=1,\ldots,k$,
	\item[(ii)] $\Delta_l$ is a subset of $M$ for every $l=1,\ldots,k$ and
	\item[(iii)] $((\rchi_{\Delta_1},m_{j_1}),\ldots,(\rchi_{\Delta_k},m_{j_k}))\in\tT_\xi$.
\end{itemize}
Note that $T_\xi(M)$, equipped with the initial segment order, is a well-founded infinite branching tree of height $\omega^\xi$.

\begin{prop}
	Let $M$ be an infinite subset of the naturals and $(x_1,\ldots,x_k)$ be any node of	 $T_\xi(M)$, then $\|x_1+\ldots+x_k\|_W\le 3$.
\end{prop}
\begin{proof}
	Let $f\in W$ with a tree analysis $(f_\alpha)_{\alpha\in \mathcal{A}}$. Observe that there exists at most one $1\le l_0\le k$ such that there is an $\alpha\in \M^1_f$ with $w(f_\alpha)=m_{j_{l_0}}$ and $\supp f_a\cap \Delta_{l_0}$ is non-empty. Indeed, suppose that there exist $1\le l_1<l_2\le k$ and $\alpha_1,\alpha_2\in \M^1_f$ with $w(f_{\alpha_1})=m_{j_{l_1}}$, $w(f_{\alpha_2})=m_{j_{l_2}}$, $\supp f_{\alpha_1}\cap \Delta_{l_1}\neq\emptyset$ and $\supp f_{\alpha_2}\cap \Delta_{l_2}\neq\emptyset$. Then since $\{t_{f_{\alpha}}:\alpha\in \M^1_f\}$ is essentially incomparable and $m_{j_{l_1}}<_{\W_\xi} m_{j_{l_2}}$ we have that $\Delta_{l_1}<\Delta_{t_{f_{\alpha_1}}}=\supp f_{\alpha_1}$ which is a contradiction.
	
	Therefore, for any $l\neq l_0$, we have that $w(f_\alpha)\neq m_{j_l}$ for every $\alpha\in\M^1_f$ and the previous proposition yields that $|f(x_l)|<2m^{-1}_{j_l}$. Moreover, Proposition \ref{special convex combination seminormalized} yields that $|f(x_{l_0})|\le 1+m^{-1}_{j_{l_0}}$ and hence we conclude that
	\[
	|f(x_1+\cdots+x_k)|\le 1+2\sum_{l=1}^k\frac{1}{m_{j_l}}\le 3.
	\]
\end{proof}

The previous proposition and the fact that the tree $T_\xi(M)$ is of height $\omega^\xi$ yield the following result.

\begin{thm}
	For every $M\in[\N]^\infty$, the space $T^\xi_{ess\-inc}$ contains a $c_0$-tree of height $\omega^\xi$, supported by $(e_j)_{j\in M}$. In particular, the space generated by $(e_j)_{j\in M}$ is not Asymptotic $\ell_1$.
\end{thm}

\begin{rem}
	There exist modifications of the ground set $G$ that yield, for any $1<p<\infty$, a space, as in the previous theorem, that contains $\ell_p$-trees instead of $c_0$-trees.
\end{rem}

\begin{thm}
	The space $T^\xi_{ess\-inc}$ is reflexive.
\end{thm}
\begin{proof}
Note that since $\T_\xi$ is a countable compact space with respect to the pointwise convergence topology, the completion of $c_{00}(\N)$ with respect to $\|\cdot\|_G$ is embedded in $C[\T_\xi]$, i.e., the space of all continuous real functions on $\T_\xi$, and hence is $c_0$-saturated. Furthermore, $T^\xi_{ess\-inc}$ admits a boundedly complete basis and therefore does not contain $c_0$. The above imply that the identity operator $Id: (c_{00}(\N),\|\cdot\|_W)\to (c_{00}(\N),\|\cdot\|_G)$ is strictly singular and hence for any normalized block sequence $(x_j)_j$ in $T^\xi_{ess\-inc}$ there exists a subsequence $(x_j)_{j\in M}$ such that $\lim_{j\in M}\|x_j\|_G=0$. The remainder of the proof is identical to the last paragraph of Proposition \ref{section 1 reflexivity}.
\end{proof}

%---------------------------------------------------------------------------------            \ell_p               -------------------------------------------

\section{More non-asymptotic $\ell_p$ spaces with uniformly unique $\ell_p$ joint spreading models}
In this final section we show that, for every  $1<p<\infty$, there is a reflexive Banach space that admits a uniformly unique $\ell_p$ asymptotic model whereas it is not Asymptotic $\ell_p$. This was also observed in \cite[Section 7.2]{BLMS} for a slightly different type of spaces. We show this for a class of spaces very similar to those defined in \cite[Example 4.3]{OS3}.

\begin{dfn}
Let $1<p<\infty$ and denote its conjugate by $q$, i.e., $p^{-1}+q^{-1}=1$. Fix a countable ordinal $\xi$ and define the following norming sets on $c_{00}(\N)$.
\[
G_1^{\xi}=\Big\{\sum_{i\in S}\epsilon_ie^*_i:S\text{ is a segment of }\T_\xi\text{ and }\epsilon_i=\pm1 \Big\}
\]
\begin{align*}
G_{1,p}^{\xi}=\Big\{\sum_{i=1}^mb_if_{i}:m\in\N,\; \sum_{i=1}^m|b_i|^q\le1,&\; f_i\in G^{\xi}_1\text{ for }i=1,\ldots,m\text{ and} \\\supp f_1,&\ldots,\supp f_m\text{ are pairwise disjoint} \Big\}.
	\end{align*}
Denote by $JT^\xi_{1,p}$ the completion of $c_{00}(\N)$ with respect to the norm induced by the norming set ${G^{\xi}_{1,p}}$.
\end{dfn}

We start with some necessary remarks on the above norming sets and a Ramsey type result.

\begin{rem}\label{remark comb lem pointwise limit}
	Let $(f_j)_j$ be a sequence in $G_{1}^{\xi}$ with $f_j=\sum_{i \in S_j}\epsilon^j_ie^*_i$, $j\in\N$, and for each $i,j\in\N$, define $\epsilon_j(i)=\epsilon^i_j$ if $i\in S_j$ and $\epsilon_j(i)=0$ otherwise. Passing to a subsequence, we may assume that $(S_j)_j$ converges pointwise to a segment $S$, since $\T_\xi$ is well-founded, and that $(\epsilon_j)_j$ also converges to some  $\epsilon$ in $\{-1,1\}^\N$. Then, clearly, $(f_j)_j$ converges pointwise to $f=\sum_{i\in S}\epsilon(i)e^*_i$ and $f$ is in $G^\xi_1$.
\end{rem}

\begin{rem}\label{remark inequality norm <= 1}
	Let $x$ be a normalized vector in $JT^\xi_{1,p}$ with finite support.
	\begin{itemize}
		\item[(i)] If for some $\varepsilon>0$ there is a family  $\{f_i\}_{i\in I}$ in $G_{1}^{\xi}$ whose members have pairwise disjoint supports and $|f_{i}(x)|\ge\varepsilon$ for all $i\in I$, then  $\#I\le\varepsilon^{-p}$.
		\item[(ii)] Let $f_1,\ldots,f_m\in G_1^{\xi}$ have pairwise disjoint supports and $\supp f_i \subset\range (x)$  for $i=1,\ldots,m$. Then, for any choice of scalars $b_1,\ldots,b_m$, we have that
		\[
		\Big|\sum_{i=1}^mb_if_{i}(x)\Big|^q\le\sum_{i=1}^m|b_i|^q.
		\]
	\end{itemize}
\end{rem}

\begin{dfn}
	We call a family $(F_j)_j$ of finite subsets of $JT^\xi_{1,p}$ a normalized block family if for any choice of $x_j\in F_j$, $j\in\N$, the sequence $(x_j)_j$ is block and $\|x\|=1$ for any $x\in F_j$ and $j\in\N$. Moreover, for such a family, define $M(F_j)=\max\{\supp x:x\in F_j\}$ and $r(F_j)=\#(M(F_{j-1}),M(F_j)]$, where $M(F_0)=0$.
\end{dfn}

\begin{lem}\label{comb lem ell2 subsequence}
	Let $(F_j)_j$ be a normalized block family in $JT^\xi_{1,p}$ with $\sup_j\# F_j<\infty$. Then, for every $\varepsilon>0$ and $n_0\in \N$, there is an $L\in[\N]^\infty$ such that, for any segment $S$ of $\T_\xi$ with $\min S\le n_0$ and any $f\in G_1^{\xi}$ with $\supp f=S$, there is at most one $j\in L$ with the property that $|f(x)|\ge\varepsilon$ for some $x\in F_j$.
\end{lem}
\begin{proof}
	For a segment $S$ of $\T_\xi$, let $G_S$ denote the set of all $f\in G_1^{\xi}$ with $\supp f=S$. If the conclusion is false for some $\varepsilon>0$ and $n_0\in\N$, then using Ramsey Theorem from \cite{Ra}, there exists an $L\in[\N]^\infty$ such that, for any $i<j$ in $L$, there is a segment $S_{ij}$ with $\min S_{ij}\le n_0$, a functional $f_{ij}\in G_{S_{ij}}$ and $x_{ij}\in F_i$, $y_{ij}\in F_j$ such that $|f_{{ij}}(x_{ij})|\ge\varepsilon$ and $|f_{{ij}}(y_{ij})|\ge\varepsilon$. Assume for convenience that $L=\N$. Since $\#\sup_j F_j<\infty$, using the pigeon hole principle and a diagonal argument we may assume that there exist sequences $(x_j)_j$, $(y_j)_j$ such that $x_j,y_j\in F_j$ and, for every $i<j\in\N$, a segment $S_{ij}$ of $\T_\xi$ with $\min S_{ij}\le n_0$ and an $f_{ij}\in G_{S_{ij}}$ such that $|f_{{ij}}(x_i)|\ge\varepsilon$ and $|f_{{ij}}(y_j)|\ge\varepsilon$.
	
	For each $i<j<k$ in $\N$, define $S_{ijk}= S_{ik}\cap S_{jk}\cap \range(y_k)$. Once more, using Ramsey theorem and passing to a further infinite subset, we may assume that $S_{ijk}$ is either always empty or always non-empty for every $i<j<k$ in $\N$. Item (i) of Remark \ref{remark inequality norm <= 1} and the fact that $\|y_k\|=1$ for all $k\in\N$ rules out the first case and hence $S_{ijk}\neq\emptyset$ for all $i<j<k$ in $\N$. This in particular implies that if we fix $i<j_1<k$ and $i<j_2<k$, then $S_{ij_1k}|_{[n_0,m(x_k))}=S_{ij_2k}|_{[n_0,m(x_k))}$. For any $j\in\N$, take an arbitrary $i$ with $1<i<j$ and set $S_{j}=S_{1ij}|_{[n_0,m(x_j))}$. Then we conclude that, for any $j\in \N$, there is an $f_j\in G_{S_j}$ such that $|f_j(x_i)|\ge\varepsilon$ for all $i<j$, where $\min S_j\le n_0$. This is a contradiction, since Remark \ref{remark comb lem pointwise limit} implies that there exists an $f\in G^{\xi}_1$ with the property that $|f(x_j)|\ge\varepsilon$ for all $j\in\N$, whereas $\supp f$ is finite since $\T_\xi$ is well-founded.
\end{proof}

\begin{lem}\label{no2 comb lem ell2 subsequence}
	Let $\varepsilon>0$ and $(F_j)_j$ be a normalized block family in $JT^\xi_{1,p}$ with $\sup_j\# F_j<\infty$. Then there exists a strictly increasing sequence $(n_j)_j$ of naturals and a decreasing sequence $(\varepsilon_j)_j$ of positive reals such that
	\begin{enumerate}
		\item[(i)] for every $j\in \N$, every segment $S$ of $\T_\xi$ with $\min S\le M(F_{n_j})$ and $f\in G^{\xi}_1$ with $\supp f=S$, there exists at most one $j'>j$ such that $|f(x)|\ge\varepsilon_j$ for some $x\in F_{n_{j'}}$ and
		\item[(ii)] $\sum_{j=1}^\infty r(F_{n_j})\sum_{i=j}^\infty(i+1)\varepsilon_i<\varepsilon$.
	\end{enumerate}
\end{lem}
\begin{proof}
	Let $(\delta_j)_j$ be a sequence of positive reals such that $\sum_{j=1}^\infty\delta_j<\varepsilon$. We will construct $(n_j)_j$ and $(\varepsilon_j)_j$ by induction, along with a decreasing sequence $(L_j)_j$ of infinite subsets of $\N$. Set $n_1=1$ and $L_1=\N$ and choose $\varepsilon_1>0$ such that $2r(F_1)\varepsilon_1<\delta_1$. Suppose that $n_1,\ldots,n_j$, $\varepsilon_1,\ldots,\varepsilon_j$ and $L_1,\ldots,L_j$ have been chosen for some $j$ in $\N$. Then, the previous lemma yields an $L_{j+1}\in[L_{j}]^\infty$ such that, for every segment $S$ of $\T_\xi$ with $\min S\le M(F_{n_j})$ and every $f\in G_1^{\xi}$ with $\supp f=S$, there is at most one $j'>j$ such that $|f(x)|\ge\varepsilon_{j}$ for some $x\in F_{n_{j'}}$. Choose $n_{j+1}\in L_{j+1}$ with $n_{j+1}>n_j$  and $\varepsilon_{j+1}<\varepsilon_j$ such that
	\begin{enumerate}
		\item[(a)]	$r(F_{n_{j+1}})(j+2)\varepsilon_{j+1}<\delta_{j+1}$ and
		\item[(b)]  $r(F_{n_{k}})\sum_{i=k}^{j+1}(i+1)\varepsilon_i<\delta_k$ for all $k\le j$.
	\end{enumerate}
	It follows quite easily that $(n_j)_j$ and $(\varepsilon_j)_j$ are as desired.
\end{proof}

\begin{prop}\label{upper ell2 for ground set}
	Let $\varepsilon>0$ and $(F_j)_j$ be a normalized block family in $JT^\xi_{1,p}$ with $\sup_j\# F_j<\infty$ satisfying the following.
	\begin{enumerate}
		\item[(i)] For every $j\in \N$, every segment $S$ of $\T_\xi$ with $\min S\le M(F_{n})$ and $f\in G^{\xi}_{1}$ with $\supp f =S$, there exists at most one $j'>j$ such that $|f(x)|\ge\varepsilon_j$ for some $x\in F_{j'}$ and
		\item[(ii)] $\sum_{j=1}^\infty r(F_j)\sum_{i=j}^\infty(i+1)\varepsilon_i<\varepsilon$.
	\end{enumerate}
	Then, for every $n\in\N$, every choice of $x_1,\ldots,x_n$ with $x_j\in F_{j}$ and scalars $a_1,\ldots,a_n$, we have that
	\[
	\Big(\sum_{j=1}^n|a_j|^p\Big)^{\frac{1}{p}}\le\Big\|\sum_{j=1}^na_jx_j\Big\|\le (2^{\frac{1}{q}}+\varepsilon)\Big(\sum_{j=1}^n|a_j|^p\Big)^{\frac{1}{p}}.
	\]
\end{prop}
\begin{proof} The lower inequality follows easily from the definition of $G_{1,p}^{\xi}$. Let us first observe that if $(x_j)_j$ is a sequence with each $x_j\in F_j$, then, for any $j\in\N$ and any segment $S$ of $\T_\xi$ with $M(x_{j-1})<\min S\le M(x_j)$ and $f\in G_1^{\xi}$ with $\supp f = S$, the following hold due to (i).
	\begin{enumerate}
		\item[(a)] $\#\{i>j:|f(x_i)|\ge\varepsilon_j\}\le 1$.
		\item[(b)] $\#\{i>j:\varepsilon_{k-1}>|f(x_i)|\ge\varepsilon_k\}\le k$ for all $k>j$.
	\end{enumerate}
	Let $f=\sum_{i=1}^mb_if_i$ be in $G_{1,p}^{\xi}$ with $\supp f_i ={S_i}$, for $i=1,\ldots,m$. For each $i$, we will denote by $j_{i,1}$ the unique $1\le j\le n$ such that $M(x_{j_{i,1}-1})<\min S_i\le M(x_{j_{i,1}})$ and by $j_{i,2}$ the unique, if there exists, $j_{i,1}<j\le n$ such that $|f_{i}(x_{j_{i,2}})|\ge\varepsilon_{j_{i,1}}$. Denote by $f_{i,1}$ the restriction of $f_i$ to ${\range(x_{j_{i,1}})}\cap\range(x_{j_{i,2}})$ and set $f_{i,2}=f_i-f_{i,1}$ for $i=1,\ldots,m$, and $I_j=\{i:j=j_{i,1}\text{ or }j=j_{i,2}\}$ for $j=1,\ldots,n$. Note that, due to (a), each $i$ appears in $I_j$ for at most two $j$ and hence $\sum_{j=1}^n\sum_{i\in I_j}|b_i|^q\le 2$.\linebreak We thus calculate applying item (ii) of Remark \ref{remark inequality norm <= 1}
	\begin{align*}	
	\sum_{i=1}^mb_if_{i,1}\Big(\sum_{j=1}^na_jx_j\Big)&=\sum_{j=1}^na_j\sum_{i\in I_j}b_if_{i,1}(x_j)\le\Big(\sum_{j=1}^n|a_j|^p \Big)^{\frac{1}{p}}\Big(\sum_{j=1}^n\big|\sum_{i\in I_j}b_if_{i,1}(x_j)\big|^q \Big)^{\frac{1}{q}}\\&\le\Big(\sum_{j=1}^n|a_j|^p \Big)^{\frac{1}{p}}\Big(\sum_{j=1}^n\sum_{i\in I_j}|b_i|^q \Big)^{\frac{1}{q}}\le2^{\frac{1}{q}}\Big(\sum_{j=1}^n|a_j|^p \Big)^{\frac{1}{p}}.
	\end{align*}
	Finally, for each $j\in\N$, set $G_j=\{i:M(x_{j_{i,1}-1})<\min S_i\le M(x_{j_{i,1}})\}$. Note that, as follows from (b), $\#G_j\le r(F_j)$ and $|f_{i,2}(\sum_{k=1}^nx_k)|<\sum_{k=i}^\infty(k+1)\varepsilon_k$ for any $i\in G_j$. Hence (ii) yields that $\sum_{i=1}^m|f_{i,2}(\sum_{k=1}^nx_k)|<\varepsilon$ and we conclude that
	\[
	\Big|\sum_{i=1}^mb_if_{i,2}\Big(\sum_{j=1}^na_jx_j\Big)\Big|=\Big|\sum_{j=1}^na_j\sum_{i=1}^mb_if_{i,2}(x_j)\Big|<\varepsilon\Big(\sum_{j=1}^n|a_j|^p \Big)^{\frac{1}{p}}
	\]
	which along with the above calculation yield the desired result.
\end{proof}

\begin{prop}
	The space $JT^\xi_{1,p}$ admits a uniformly unique joint spreading model with respect to $\mathscr{F}_b(JT^\xi_{1,p})$,  equivalent to the unit vector basis of $\ell_p$.
\end{prop}
\begin{proof}
	Let $(x^1_j)_j,\ldots,(x^l_j)_j$ be normalized block sequences in $JT^\xi_{1,p}$ and let $\varepsilon>0$.\linebreak Applying Lemma \ref{no2 comb lem ell2 subsequence} and passing to a subsequence, we may assume that $F_j=\{x^i_j:i=1,\ldots,l\}$ is a normalized block family in $JT^\xi_{1,p}$ satisfying items (i) and (ii) of Proposition \ref{upper ell2 for ground set}.  Then, for every $k\in\N$, every  $s=(s_i)_{i=1}^l$ in $S$-$Plm_l([L]^k)$ and any choice of scalars $(a_{ij})_{i=1,j=1}^{l,k}$, we calculate
	\[
	\Big(\sum_{i=1}^l\sum_{j=1}^k|a_{ij}|^p\Big)^{\frac{1}{p}}\le\Big\|\sum_{i=1}^l\sum_{j=1}^ka_{ij}x^i_{s_i(j)}\Big\|\le (2^{\frac{1}{q}}+\varepsilon)\Big(\sum_{i=1}^l\sum_{j=1}^k|a_{ij}|^p\Big)^{\frac{1}{p}}.
	\]
	A diagonal argument then yields that there exists $L\in[\N]^\infty$ such that $((x^i_j)_{j\in L})_{i=1}^l$ generates a joint spreading model $2^\frac{1}{q}$-equivalent to the unit vector basis of $\ell_p$.
\end{proof}

\begin{prop}
			The space $JT^\xi_{1,p}$  is reflexive.
\end{prop}
\begin{proof}
	Note that the unit vector basis of $c_{00}(\N)$ forms a boundedly complete unconditional Schauder basis for $JT^\xi_{1,p}$, that is, it does not contain $c_0$. Moreover, Proposition \ref{upper ell2 for ground set} yields that it does not contain $\ell_1$ and hence Theorem 2 from \cite{J1} yields the desired result.
\end{proof}

\begin{prop}
	The space $JT^\xi_{1,p}$ is not Asymptotic $\ell_p$.
\end{prop}
\begin{proof}
	Suppose that $JT^\xi_{1,p}$ is $C$-Asymptotic $\ell_p$ and let $n\in\N$ be such that $C\le n^\frac{1}{q}$. Then, following the same arguments as in Proposition \ref{NOT Asymptotic l1}, in the final outcome of $G(n,p,C)$ we, as player (V), have chosen elements of the basis $e_{j_1},\ldots,e_{j_n}$ such that $\{j_1,\ldots,j_n\}$ is a segment of $\T_\xi$ and hence $\{e_{j_1},\ldots,e_{j_n}\}$ is isometric to $\ell_1^n$. We then calculate
	\[
	\Big\| {n^{-\frac{1}{p}}}\sum_{i=1}^n e_{j_i}\Big\|=\Big\| {n^{-\frac{1}{p}}}\sum_{i=1}^n e_{j_i}\Big\|_{G_1^\xi}=n^\frac{1}{q}
	\]
	whereas, since $JT^\xi_{1,p}$ is $C$-Asymptotic $\ell_p$, we have that
	\[
	\Big\| {n^{-\frac{1}{p}}}\sum_{i=1}^n e_{j_i}\Big\|\le C
	\]
	and this is a contradiction.
\end{proof}

\begin{rem}
	We may also define a conditional version of $JT^\xi_{1,p}$,  denoted as $JT^\xi_{p}$, by replacing the norming set $G_1^\xi$ with
	\[
	G^\xi_{sum}=\Big\{\sum_{i \in S}e^*_i:S\text{ is a segment of } \T_\xi \Big\}.
	\]
	Note that the above results hold for $JT^\xi_{p}$.  For the reflexivity part, notice that it suffices to show that $(e_j)_j$ is shrinking for $JT^\xi_{p}$.  If not, then there is an $x^*\in (JT^\xi_{p})^{*}\setminus\overline{\text{span}}\{e^*_j \}_{j=1}^\infty$ and  an $x^{**}\in (JT^\xi_{p})^{**}$ with $x^{**}(e^*_j)=0$ for all $j\in\N$ and $x^{**}(x^*)=1$. Then, from Odell-Rosenthal Theorem \cite{OR} and the fact that $x^{**}(e^*_j)=0$, $j\in\N$, we may find a seminormalized block sequence $(x_j)_j$ in $JT^\xi_{p}$ with $w^*$-$\lim_jx_j=x^{**}$ and, passing to a subsequence, we may assume that it also satisfies items (i) and (ii) of Proposition \ref{upper ell2 for ground set} for some $\varepsilon>0$. Since $x^{**}(x^*)=1$, there exists $n_0\in\N$ such that $x^*(x_n)\ge1/2$ for all $n\ge n_0$. Then, for $k\in\N$ such that $(2^\frac{1}{q}+\varepsilon)k^{-\frac{1}{q}}<1/2$, Proposition \ref{upper ell2 for ground set} yields that
		\[
		x^*\Big( \frac{x_{n_0+1}+\ldots+x_{n_0+k}}{k}\Big)\le(2^\frac{1}{q}+\varepsilon)k^{-\frac{1}{q}}
		\]
		which is a contradiction.
	\end{rem}
	
	\begin{rem}
		Note that by replacing the norming set $G_1^\xi$ with
		\[
		G_r^{\xi}=\Big\{\sum_{i\in S}b_ie^*_i:S\text{ is a segment of }\T_\xi\text{ and }\sum_{i \in S}|b_i|^{r'}\le1 \Big\}
		\]		
		where $r^{-1}+r'^{-1}=1$ and $1<r<p$, we define the spaces $JT^\xi_{r,p}$  whose norm is described in \eqref{jtqp}. These spaces are also reflexive, admit a  unique $\ell_p$ asymptotic model and are not Asymptotic $\ell_p$.
	\end{rem}

	\begin{rem}
	The approach used in \cite{BLMS} can be used to show that the spaces ${JT}^\xi_{r,p}$ and ${JT}^\xi_{p}$ have the property that any joint spreading model generated by an array of weakly null sequences is isometrically equivalent to the unit vector basis of $\ell_p$. That approach provides less insight and has no potential to apply to cases with a non-isometric result, e.g., the space from Section \ref{ell1section}.
	\end{rem}
	
\section*{Acknowledgement}
The authors would like to thank the anonymous referees whose useful remarks helped us identify and correct a mistake in the original proof of Lemma \ref{combinatorial lemma essentially incomparable}.

\end{document}